\documentclass[12pt,a4paper,final]{amsart}

\pdfoutput=1
\usepackage{amsmath,amsthm,amssymb}
\usepackage{amsfonts}
\usepackage[whole]{bxcjkjatype}
\usepackage[top=25truemm,bottom=25truemm,left=30truemm,right=30truemm]{geometry}
\usepackage{newtxtext}
\usepackage[varg]{newtxmath}
\usepackage[mathscr]{eucal}
\usepackage{bxpapersize}
\usepackage{ascmac}
\usepackage{epic, eepic}
\usepackage{graphicx}
\usepackage{mathtools}
\usepackage{indentfirst}
\usepackage{cases}
\usepackage{color}
\usepackage{url}
\usepackage[]{multicol}
\usepackage{bm}
\usepackage[all]{xy}
\usepackage[noadjust]{cite}
\usepackage{showkeys}
\usepackage{enumerate}

\theoremstyle{plane}

\newtheorem{thm}{Theorem}[section]
\newtheorem{prop}[thm]{Proposition}
\newtheorem{lem}[thm]{Lemma}
\newtheorem{cor}[thm]{Corollary}
\theoremstyle{definition}
\newtheorem{dfn}[thm]{Definition}
\newtheorem{ex}[thm]{Example}

\theoremstyle{remark}
\newtheorem{rem}[thm]{Remark}
 \newtheorem*{ack}{Acknowledgements}

\newcommand{\rank}{\operatorname{rank}}

\newcommand{\im}{\operatorname{Im}}

\newcommand{\R}{\mathbb{R}}
\newcommand{\bS}{\mathbb{S}}

\newcommand{\bV}{\bm{V}}
\newcommand{\e}{\bm{e}}

\newcommand{\bx}{\bm{x}}

\newcommand{\F}{\mathcal{F}}

\newcommand{\nr}{NR}

\renewcommand{\phi}{\varphi}

\newcommand{\what}{\widehat}

\newcommand{\wtil}{\widetilde}
\newcommand{\cal}[1]{\mathcal{#1}}

\newcommand{\inner}[2]{\left\langle{#1},{#2}\right\rangle}

\numberwithin{equation}{section}

\title[Singular loci of normal congruence of frontals]
{On surfaces obtained as singular loci of normal congruence of frontals with pure-frontal singular points}
\author[S. P. dos Santos]{Samuel P. dos Santos}
\address[S.~P. dos Santos]{Departmento de Matem\'{a}tica, 
Universidade Estadual Paulista (UNESP),
R.~Crist\'{o}v\~{a}o Colombo, 2265, Jardim Nazareth, 
15054-000 S\~{a}o Jos\'{e} do Rio Preto, SP, Brazil}
\email{samuelp.santos@hotmail.com}
\author[K. Teramoto]{Keisuke Teramoto}
\address[K. Teramoto]{Department of Mathematics, 
Hiroshima University, Higashi-Hiroshima 739-8526, Japan}
\email{kteramoto@hiroshima-u.ac.jp}

\thanks{The first author was partly supported by the Grant 2018/17712-7, 
 S\~{a}o Paulo Reserch Foundation (FAPESP). 
	The second author was partly supported by JSPS KAKENHI Grant Numbers 
	JP19K14533, JP22K13914 and JP20H01801. The authors were partly supported by 
the Japan-Brazil bilateral project JPJSBP1 20190103.}
\date{\today}

\begin{document}

\begin{abstract}
	We study singularities and geometric properties of surfaces given by the singular loci of normal congruence of frontals with 
	pure-frontal singular points.
	These surfaces consist of the normal ruled surface and focal surfaces of the initial frontal. 
	For the normal ruled surface, we give characterizations of singularities in terms of geometric invariants of 
	the initial frontal defined on the set of singular points. 
	For focal surfaces, we show relation between certain singularities of them and geometric property of the given frontal. 
	Moreover, we consider behavior of Gaussian curvature of focal surfaces of frontal with a $5/2$-cuspidal edge. 
\end{abstract}

\maketitle

\section{Introduction}
In differential geometry of surfaces in the Euclidean $3$-space $\R^3$, focal surfaces are classical objects. 
Although initial surfaces do not have any singular points, 
their focal surfaces have singular points in general. 
Porteous \cite{porteous0} (see also \cite{porteous}) studied focal surfaces of regular surfaces using singularity theory techniques. 
In the study, he defined a notion of (higher order) {\it ridge points} on a surface 
and showed relation between ridge points and types of singularities on focal surfaces. 
In contrast, Bruce and Wilkinson \cite{bw-fold} investigated the folding map of surfaces 
and they introduced the notion of {\it sub-parabolic points} on surfaces, 
which correspond to parabolic points on focal surface (see also \cite{morris,porteous}). 
Although definitions of these concepts are simple, 
they were not studied deeply before their investigations. 

On the other hand, there are several articles treating 
surfaces with singular points from the differential geometric viewpoint recently 
(cf. \cite{fukui,ft,ft-onepara,hhnuy,hhnsuy,suy_front,ms,msuy,os,hs,saji_santos}). 
In particular, the study of classes called {\it frontals} and {\it fronts} has been attractive. 
Although the first fundamental form (or the induced metric) 
of a surface in the Euclidean $3$-space $\R^3$ is degenerate at singular points,  
frontals or fronts admit smooth unit normal vector even at singular points. 
Since a unit normal vector field of a given frontal can be taken smoothly, 
we can consider {\it normal congruence} of  the frontal naturally. 
It is known that singular loci of normal congruence form {\it focal surfaces} (or {\it caustics}) 
of the initial surface (\cite{arnold,ist-congru}). 
As mentioned above, to investigate focal surfaces of frontals, 
it is expected that we might obtain new geometrical properties for frontals. 

In this paper, we study singularities and geometric properties of surfaces, 
which are arose from the singular loci of normal congruence of frontals 
with {\it pure-frontal singular points} (see Section \ref{sec:frontal}). 
For regular surfaces and fronts with certain singularities, the singular loci corresponds to only focal surfaces (\cite{ist-congru,tera1}). 
However, in the case of frontals, we obtain an additional surface (say the {\it normal ruled surface}, see Section \ref{sec:congruence}). 
This might be a characteristic phenomenon of a frontal but not a front.
Normal ruled surfaces are ruled surfaces whose base curves are the singular loci of given frontals 
and the direction of generators is corresponding unit normal vector. 
We clarify relation between types of singularities of normal ruled surfaces 
and geometrical properties of given frontal surface (Theorems \ref{thm:sing_NR} and \ref{thm:nondev-NR}). 
Moreover, we investigate singularities and geometrical properties of focal surfaces of frontals in Section \ref{sec:focal}. 
We give geometric characterization for focal surfaces to be singular 
at the same singular point of the initial frontal (Proposition \ref{prop:sing_mean}). 
By this result, we find that if the initial frontal has a $5/2$-cuspidal edge, 
then both focal surfaces are regular at that point (Corollary \ref{cor:sing_sing}). 
When focal surfaces have singular points, 
we characterize those points to be of the first kind or of the second kind 
by geometric properties of the initial frontal (Proposition \ref{prop:sk}). 
In particular, if a singular point of focal surfaces is of the first kind, 
we show condition for the point to be a cuspidal cross cap 
in terms of geometric invariants of the initial frontal (Theorem \ref{thm:ccr}). 
Furthermore, we show that the singular set of the initial frontal surface consisting of pure-frontal singular points 
is also the set of pure-frontal singular points of corresponding focal surfaces under certain geometrical properties (Theorem \ref{thm:pure}).  
Finally, we investigate behavior of the Gaussian and the mean curvature of the focal surfaces 
at a $5/2$-cuspidal edge of the initial frontal (Theorem \ref{thm:curvature} and Proposition \ref{prop:Cj-mean}).

\section{Frontal surfaces}\label{sec:frontal}
We recall some notions and properties of frontal surfaces. 
For detailed explanations, see \cite{arnold,ishikawa_frontal,ishikawa_recog,saji_S1,msuy,suy_front}.

Let $f\colon (\R^2,0)\to(\R^3,0)$ be a $C^\infty$ map germ. 
Then $f$ is said to be a {\it frontal} if there exists a $C^\infty$ map 
$\nu\colon(\R^2,0)\to \bS^2$ such that $\inner{df_q(X)}{\nu(q)}=0$ 
holds for any $q\in(\R^2,0)$ and $X\in T_q\R^2$, 
where $\bS^2$ is the standard unit sphere in $\R^3$ and $\inner{\cdot}{\cdot}$ 
is the canonical inner product of $\R^3$. 
Moreover, a frontal $f$ is called a {\it front} 
if the pair $(f,\nu)\colon(\R^2,0)\to(\R^3\times \bS^2,(0,\nu(0)))$ gives an immersion.
We call $\nu$ a {\it unit normal vector field} of $f$. 

We fix a frontal $f\colon(\R^2,0)\to(\R^3,0)$ with a unit normal vector $\nu$. 
We set a function $\lambda\colon(\R^2,0)\to\R$ by 
\begin{equation}\label{eq:lambda}
	\lambda(u,v)=\det(f_u,f_v,\nu)(u,v)\quad 
	(f_u=\partial f/\partial u,\ f_v=\partial f/\partial v),
\end{equation}
where $(u,v)$ is a local coordinate system on $(\R^2,0)$ and 
$\det$ is the determinant. 
We call $\lambda$ the {\it signed area density function}. 
A function ${\Lambda}$ is called an {\it identifier of singularities} of $f$ 
if ${\Lambda}$ is a nowhere-vanishing function multiple of $\lambda$. 
Denoting by $S(f)=\{q\in (\R^2,0)\ |\ \rank df_q<2\}$ the set of singular points of $f$, 
we see that ${\Lambda}^{-1}(0)=S(f)$. 

Assume that $0\in S(f)$ in the following. 
Then $0$ is said to be {\it non-degenerate} if $(\Lambda_u,\Lambda_v)\neq(0,0)$ at $0$. 
When $0$ is a non-degenerate singular point of $f$, 
then there exists a regular curve $\gamma\colon(\R,0)\to(\R^2,0)$ 
such that $\Lambda(\gamma)=0$ holds. 
Moreover, in such a case, since $\rank df_{0}=1$, 
there exists a never vanishing vector field $\eta$ on $(\R^2,0)$ 
such that $df(\eta)=0$ along $\gamma$. 
We call $\gamma$ and $\eta$ a {\it singular curve} and a {\it null vector field} for $f$, respectively. 
The origin $0$ is said to be a singular point of the {\it first kind} (resp. {\it second kind})
if $\gamma'=d\gamma/dt$ is linearly independent (resp. dependent) to $\eta$ at $0$. 
We note that if the singular curve $\gamma$ through $0$ 
consists of the first kind, then the singular locus $\hat{\gamma}=f\circ\gamma$ of $f$ is a regular spacial curve. 

We assume that the origin $0$ is a singular point of the first kind of 
a frontal $f\colon(\R^2,0)\to(\R^3,0)$. 
We define a function $\psi\colon(\R,0)\to\R$ by 
\begin{equation}\label{eq:psi}
	\psi(t)=\det(\hat{\gamma}'(t),\nu(\gamma(t)),d\nu_{\gamma(t)}(\eta)),
\end{equation}
where $\hat{\gamma}'=d\hat{\gamma}/dt$ and $d\nu_{\gamma(t)}(\eta)$ is the directional derivative of $\nu$ 
in the direction of a null vector field $\eta$ along $\gamma$. 

\begin{prop}[\cite{fsuy,msuy}]
	Let $f\colon(\R^2,0)\to(\R^3,0)$ be a frontal and $0$ a singular point of the first kind. 
	Then $f$ is a front at $0$ if and only if $\psi(0)\neq0$, 
	where $\psi$ is a function defined by \eqref{eq:psi}.
\end{prop}

By this fact, we call the singular point of the first kind $0$ of a frontal $f$ a {\it non-front singular point} if $\psi(0)=0$. 
\begin{dfn}[\cite{saji_tera}]
	Let $f\colon(\R^2,0)\to(\R^3,0)$ be a frontal, $\nu$ its unit normal vector 
	and $0$ is a singular point of  the first kind of $f$. 
	Let $\psi$ be a function in \eqref{eq:psi} defined along the singular curve $\gamma$ passing through $0$. 
	Then $0$ is said to be a {\it $k$-non-front singular point} of $f$ 
	if there exists a positive integer $k$ such that $\psi(0)=\psi'(0)=\cdots=\psi^{(k-1)}(0)=0$ and $\psi^{(k)}(0)\neq0$ hold; 
	a point $0$ is said to be a {\it pure-frontal singular point} if $\psi=0$ along the singular curve $\gamma$ through $0$. 
\end{dfn}

\begin{dfn}
	Let $f\colon(\R^2,0)\to(\R^3,0)$ a $C^\infty$ map germ and $0$ a singular point of $f$. 
	Then 
	\begin{itemize}
		\item $f$ at $0$ is a {\it cuspidal edge} if it is $\cal{A}$-equivalent to the germ $(u,v)\mapsto(u,v^2,v^3)$ at the origin;
		\item $f$ at $0$ is a {\it cuspidal cross cap} 
		if it is $\cal{A}$-equivalent to the germ $(u,v)\mapsto(u,v^2,uv^3)$ at the origin;
		\item $f$ at $0$ is a {\it cuspidal $S_{k\geq0}^{\pm}$ singularity} 
		if it is $\cal{A}$-equivalent to the germ $(u,v)\mapsto(u,v^2,v^3(u^{k+1}\pm v^2))$ at the origin;
		\item $f$ at $0$ is a {\it $5/2$-cuspidal edge} (or a {\it rhamphoid cuspidal edge}) 
		if it is $\cal{A}$-equivalent to the germ $(u,v)\mapsto(u,v^2,v^5)$ at the origin;
		\item $f$ at $0$ is a {\it fold singularity} 
		if it is $\cal{A}$-equivalent to the germ $(u,v)\mapsto(u,v^2,0)$ at the origin.
	\end{itemize}
	Here two map germs $f,g\colon(\R^2,0)\to(\R^3,0)$ are said to be {\it $\cal{A}$-equivalent} 
	if there exist diffeomorphism germs $S\colon(\R^2,0)\to(\R^2,0)$ on the source 
	and $T\colon(\R^3,0)\to(\R^3,0)$ on the target such that $T\circ f=g\circ S$ holds. 
\end{dfn}

\begin{figure}[htbp]
	\centering
		\begin{tabular}{c}
			
			\begin{minipage}{0.3\hsize}
				\begin{center}
					\includegraphics[clip, width=3cm]{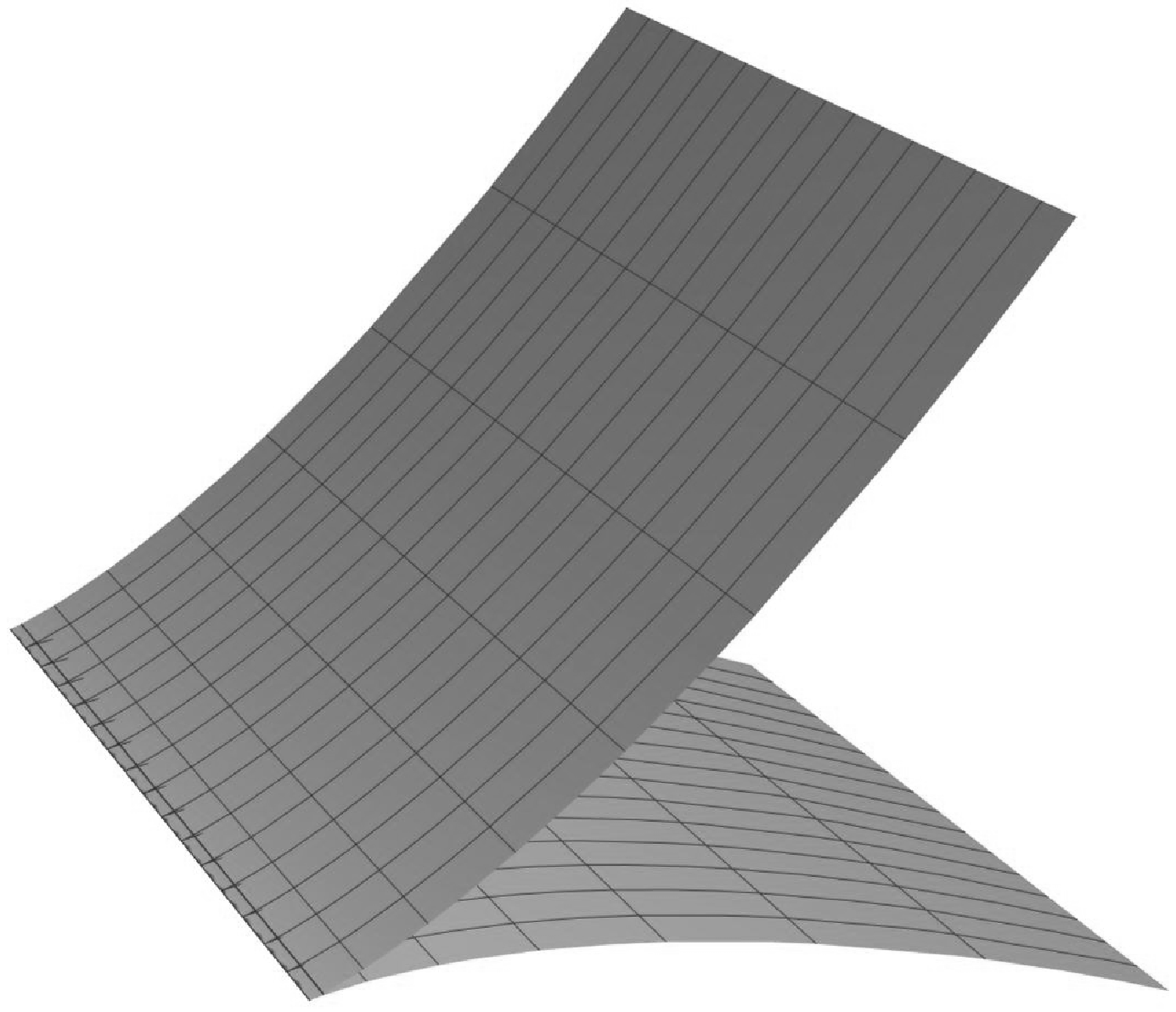}
				\end{center}
			\end{minipage}
			
			\begin{minipage}{0.3\hsize}
				\begin{center}
					\includegraphics[clip, width=3cm]{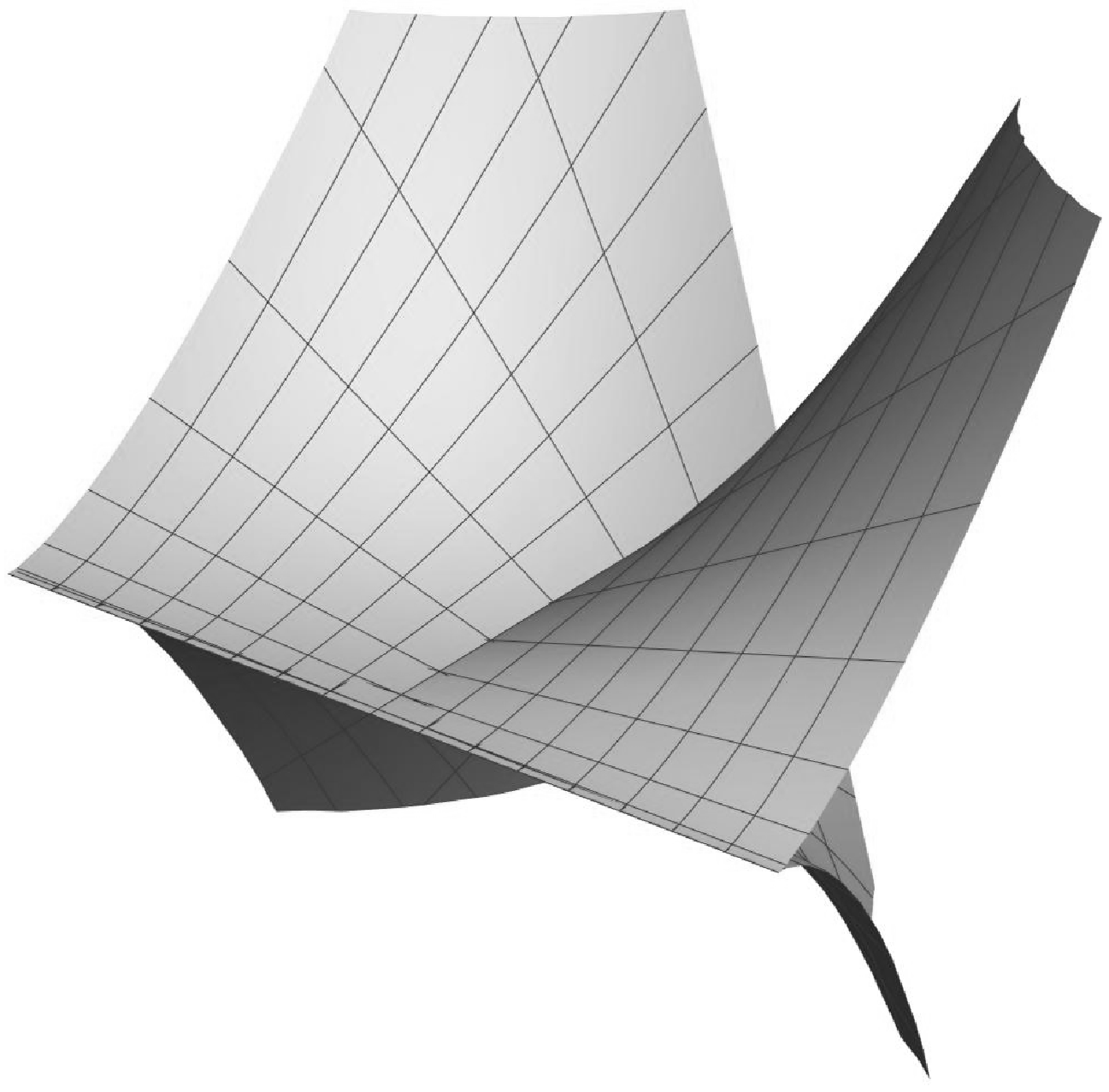}
				\end{center}
			\end{minipage}
			
			\begin{minipage}{0.3\hsize}
				\begin{center}
					\includegraphics[clip, width=3cm]{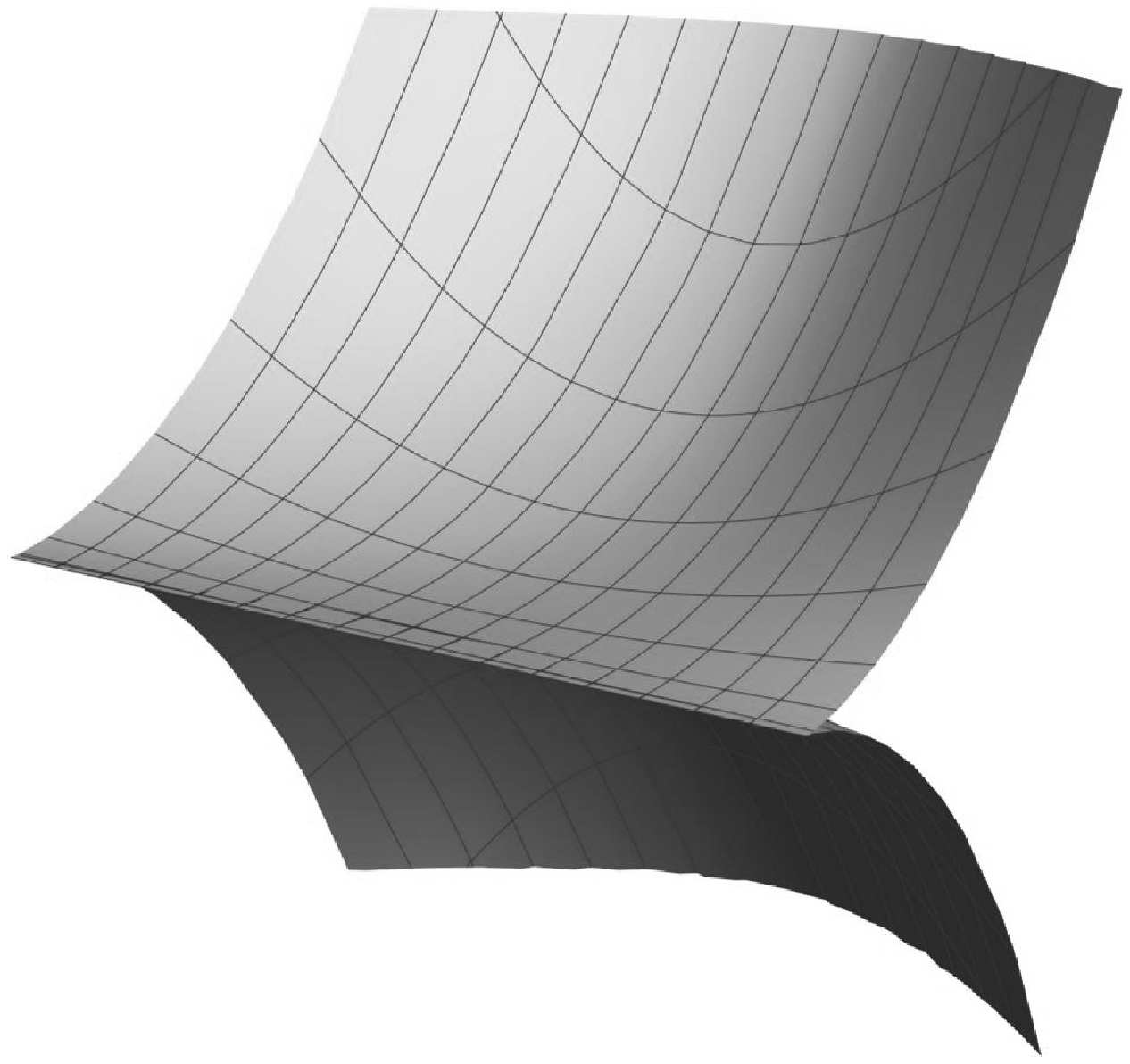}
				\end{center}
			\end{minipage}\\
			
			\begin{minipage}{0.3\hsize}
				\begin{center}
					\includegraphics[clip, width=3cm]{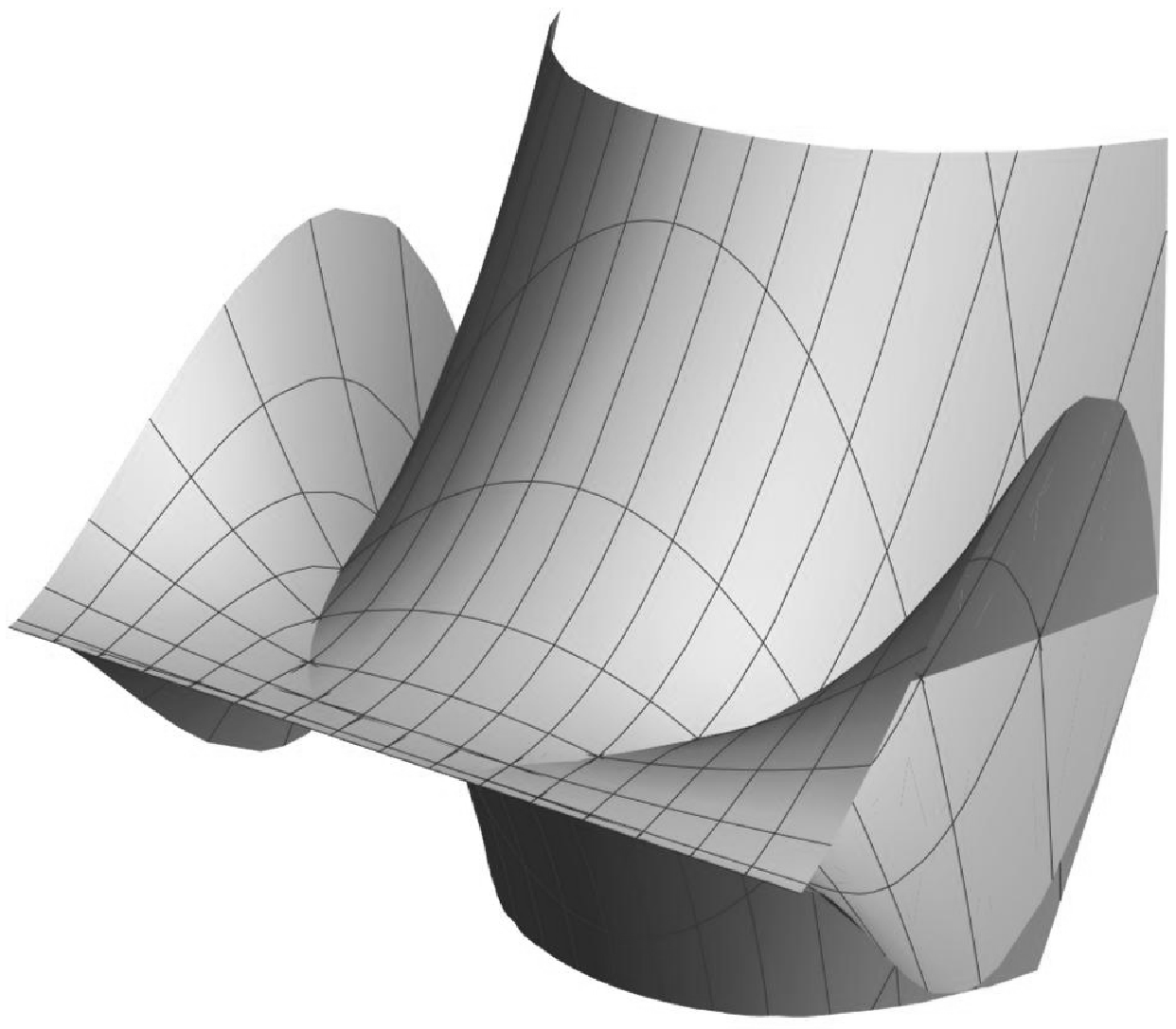}
				\end{center}
			\end{minipage}
			
			\begin{minipage}{0.3\hsize}
				\begin{center}
					\includegraphics[clip, width=3cm]{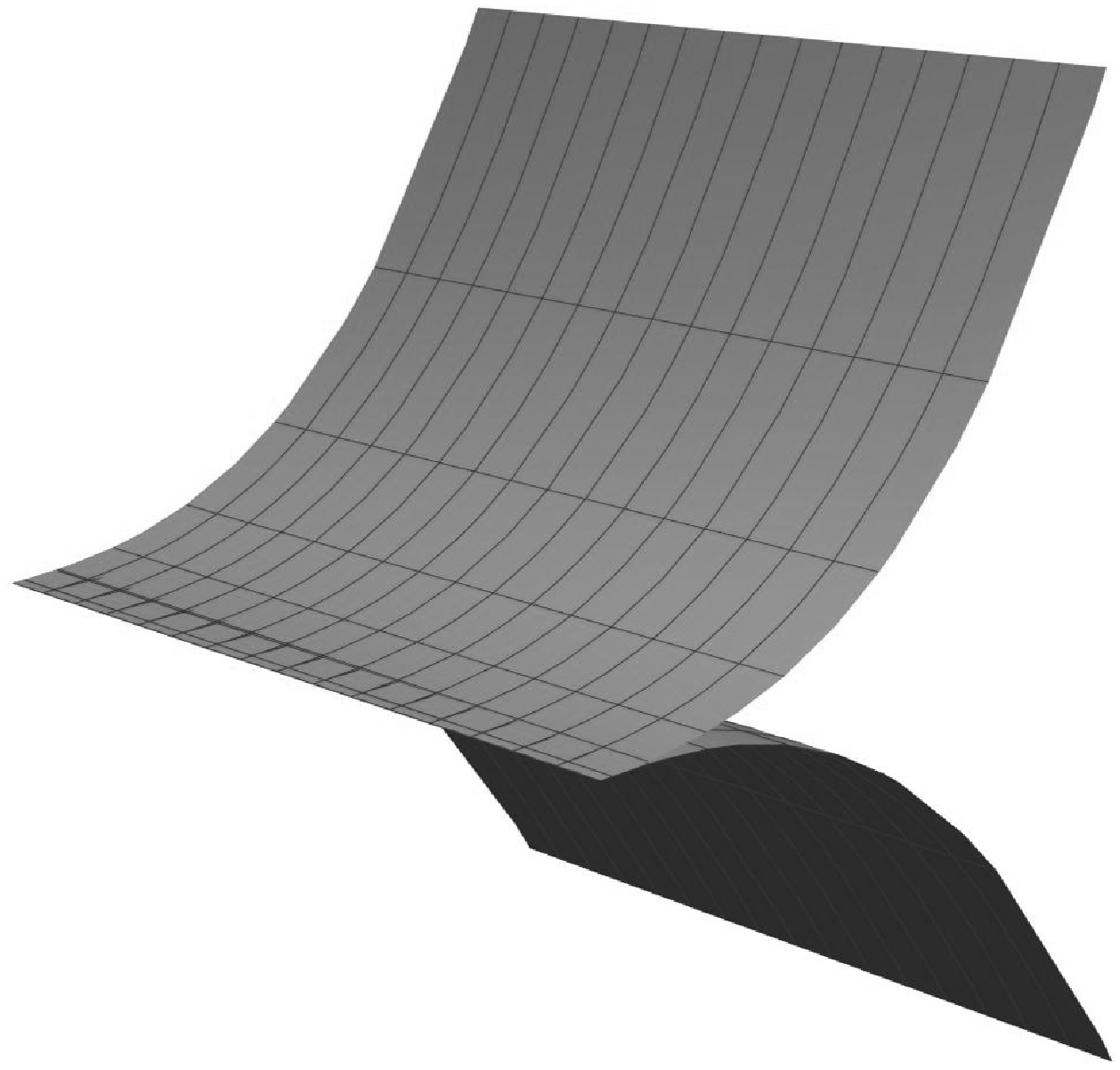}
				\end{center}
			\end{minipage}
			
			\begin{minipage}{0.3\hsize}
				\begin{center}
					\includegraphics[clip, width=3cm]{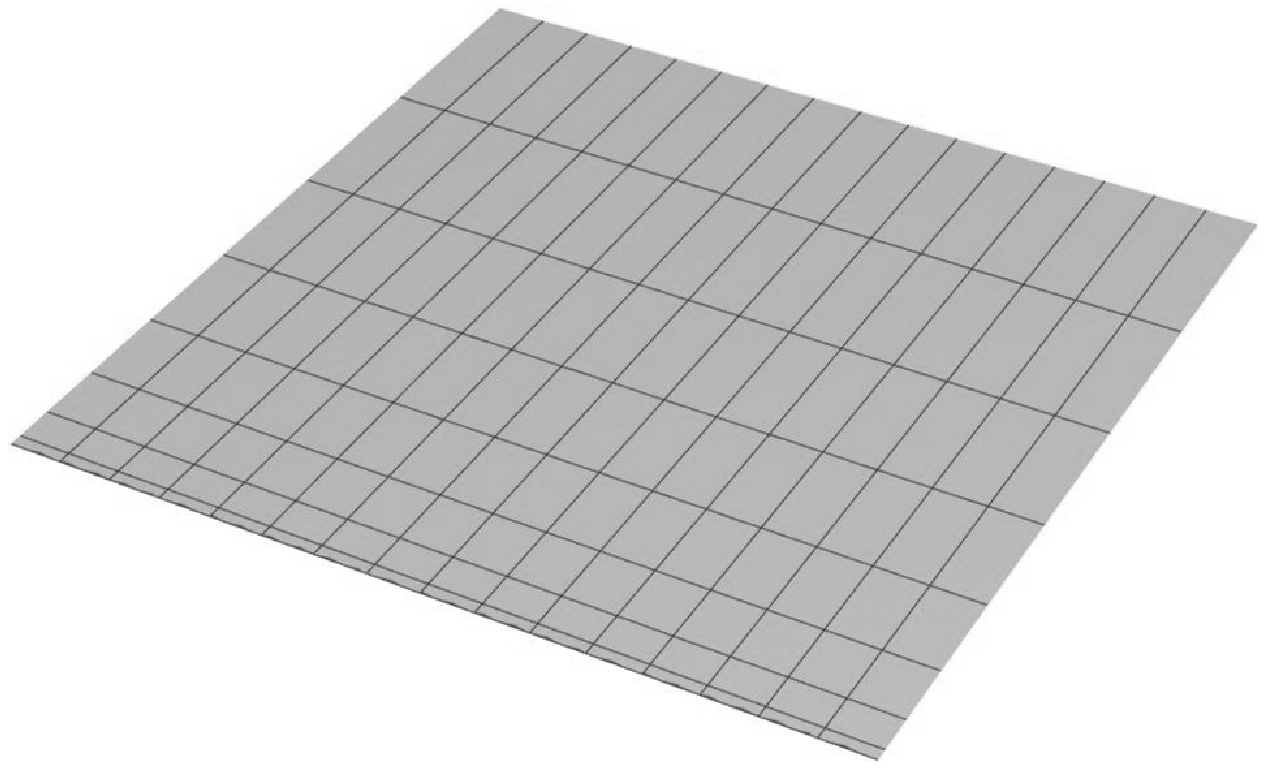}
				\end{center}
			\end{minipage}
			
		\end{tabular}
		\caption{From top left to bottom right: The images of a cuspidal edge, a cuspidal cross cap, a cuspidal $S_1^+$ singularity, 
			a cuspidal $S_1^{-}$ singularity, a $5/2$-cuspidal edge and a fold singularity. }
		\label{fig:singular}
\end{figure}

Singular points in the above definition are all singular points of the first kind. 
A cuspidal edge is a singularity of a front, but 
other singular points are examples of non-front singularities. 
Although cuspidal cross caps and cuspidal $S_k$ singularities are 
$k$-non-front singular points for some integer $k\geq1$, 
$5/2$-cuspidal edges and fold singular points are pure-frontal singularities of frontals (\cite{fsuy,hs,saji_S1,saji_tera}). 
We remark that a cuspidal $S_0$ singularity is $\cal{A}$-equivalent to a cuspidal cross cap. 
In the following, our main target is a frontal with pure-frontal singularities.

\subsection{Geometrical properties}
We recall geometrical properties of a frontal $f\colon(\R^2,0)\to(\R^3,0)$ with pure-frontal singular point $0$. 
Under this setting, 
we can take a local coordinate system $(u,v)$ on $(\R^2,0)$ (\cite{msuy}) satisfying 
\begin{itemize}
	\item the $u$-axis gives a singular curve, that is, $S(f)=\{v=0\}$,
	\item $\partial_v$ gives a null vector field $\eta$,
	\item $\{f_u,f_{vv},\nu\}$ gives an orthonormal frame along the $u$-axis.
\end{itemize}
We call this coordinate system {\it  adapted}. 

When we take an adapted coordinate system $(u,v)$ on $(\R^2,0)$, 
then there exists a map $h\colon (\R^2,0)\to\R^3$ 
such that $f_v=vh$ holds. 
Note that $h$ is perpendicular to $\nu$. 
Since $\lambda_v=\det(f_u,f_{vv},\nu)=\det(f_u,h,\nu)\neq0$ at $0$, 
the triple $\{f_u,h,\nu\}$ forms a moving frame along $f$ on $(\R^2,0)$. 
Using this frame, 
we set the following functions: 
\begin{equation}\label{eq:fundamentals}
	\begin{aligned}
		\wtil{E}&=\inner{f_u}{f_u}, & \wtil{F}&=\inner{f_u}{h}, & \wtil{G}&=\inner{h}{h},\\
		\wtil{L}&=-\inner{f_u}{\nu_u}, & \wtil{M}&=-\inner{h}{\nu_u}, & \wtil{N}&=-\inner{h}{\nu_v}.
	\end{aligned}
\end{equation}
If a frontal $f$ has a pure-frontal singular point $0$, 
then the function $\psi$ as in \eqref{eq:psi} vanishes identically along the singular curve $\gamma$ 
through $0$. 
Thus when we take an adapted coordinate system $(u,v)$, 
then $\nu_v(u,0)=0$. 
This implies that there exists a map $\nu_1\colon (\R^2,0)\to\R^3$ 
such that $\nu_v=v\nu_1$. 
Moreover, setting $\wtil{N}_1=-\inner{h}{\nu_1}$, 
the function $\wtil{N}$ as in \eqref{eq:fundamentals} can be rewritten as 
$\wtil{N}=v\wtil{N}_1$.  

\begin{lem}[\cite{saji_tera}]\label{lem:weingarten}
	Let $(u,v)$ be an adapted coordinate system around a pure-frontal singular point $0$ 
	of a frontal $f\colon (\R^2,0)\to(\R^3,0)$. 
	Then 
	\begin{equation}\label{eq:weingarten}
		\nu_u=\frac{\wtil{F}\wtil{M}-\wtil{G}\wtil{L}}{\wtil{E}\wtil{G}-\wtil{F}^2}f_u+
		\frac{\wtil{F}\wtil{L}-\wtil{E}\wtil{M}}{\wtil{E}\wtil{G}-\wtil{F}^2}h,\ 
		\nu_v=v\frac{\wtil{F}\wtil{N}_1-\wtil{G}\wtil{M}}{\wtil{E}\wtil{G}-\wtil{F}^2}f_u+
		v\frac{\wtil{F}\wtil{M}-\wtil{E}\wtil{N}_1}{\wtil{E}\wtil{G}-\wtil{F}^2}h.
	\end{equation}
\end{lem}

We next consider curvatures of a frontal $f\colon (\R^2,0)\to(\R^3,0)$ around a pure-frontal singular point $0$. 
Let us take an adapted coordinate system $(u,v)$ on $(\R^2,0)$. 
Then the Gaussian curvature $K$ and the mean curvature $H$ are expressed as 
\begin{equation}\label{eq:KH}
	K=\frac{\wtil{L}\wtil{N}_1-\wtil{M}^2}{\wtil{E}\wtil{G}-\wtil{F}^2},\quad 
	H=\frac{\wtil{E}\wtil{N}_1-2\wtil{F}\wtil{M}+\wtil{G}\wtil{L}}{2(\wtil{E}\wtil{G}-\wtil{F}^2)}
\end{equation}
on the set of regular points of $f$. 
Since $\wtil{E}\wtil{G}-\wtil{F}^2>0$, $K$ and $H$ can be defined on $(\R^2,0)$. 
In particular, these are bounded $C^\infty$ functions on $(\R^2,0)$. 
By using $K$ and $H$, we define {\it principal curvatures} $\kappa_j$ $(j=1,2)$ as follows:
\begin{equation}\label{eq:princ-curv}
	\kappa_1=H+\sqrt{H^2-K},\quad \kappa_2=H-\sqrt{H^2-K}.
\end{equation}
We remark that $H^2-K\geq0$ holds.
We also note that if the cuspidal torsion $\kappa_t$ (see Subsection \ref{subsec:geoinv}) does not vanish at $0$, 
then $0$ is a non-umbilic point, that is, $\kappa_1\neq\kappa_2$ at $0$ (\cite[Theorem 3.1]{saji_tera}).  
Assuming $\kappa_t\neq0$ at $0$, we can take {\it principal vectors} $\bm{V}_j$ $(j=1,2)$ 
relative to $\kappa_j$. 
Using the functions \eqref{eq:fundamentals}, $\bm{V}_j$ can be written as 
\begin{equation}\label{eq:princ-vect}
	\bm{V}_j=(-v(\wtil{M}-\kappa_j\wtil{F}),\wtil{L}-\kappa_j\wtil{E})\quad 
	(j=1,2)
\end{equation}
(\cite[(3.9)]{saji_tera}). 
\begin{rem}
By \eqref{eq:KH} and \eqref{eq:princ-curv}, the principal curvatures of a pure frontal can be extended as $C^\infty$ functions. If a frontal $f\colon (\R^2,0)\to(\R^3,0)$ has a $k$-non-front singular point $(k\geq0)$ at $0$, 
	then the principal curvatures cannot be extended as $C^\infty$ functions (\cite[Theorem 4.1]{saji_tera}). 
	On the other hand, if $f$ is a front and $0$ a cuspidal edge, 
	then one of the two principal curvature can be extended as a $C^\infty$ function 
	and another is unbounded near $0$ (\cite{muraume,tera0}).
\end{rem}

Using Lemma \ref{lem:weingarten}, we have the following.
\begin{lem}\label{lem:rod}
	Let $f\colon(\R^2,0)\to(\R^3,0)$ be a frontal with pure-frontal singular point $0$ which is a non-umbilic point. 
	Let $\bm{V}_j$ be principal vectors with respect to $\kappa_j$. 
	Then 
	\[d\nu(\bm{V}_j)=-\kappa_jdf(\bm{V}_j).\]
\end{lem}
\begin{proof}
	We give a proof for $j=1$. 
	Take an adapted coordinate system $(u,v)$. 
	Then the principal vector $\bm{V}_1$ is given by 
	\[\bm{V}_1=(-v(\wtil{M}-\kappa_1\wtil{F}),\wtil{L}-\kappa_1\wtil{E})\]
	(cf. \eqref{eq:princ-vect}). 
	Thus we have 
	\begin{equation}\label{eq:dfv}
	df(\bm{V}_1)=-v(\wtil{M}-\kappa_1\wtil{F})f_u+v(\wtil{L}-\kappa_1\wtil{E})h(=v\bx_1).
	\end{equation}
	On the other hand, by direct calculations with Lemma \ref{lem:weingarten}, 
	we get 
	\[
	d\nu(\bm{V}_1)=v\kappa_1(\wtil{M}-\kappa_1\wtil{F})f_u-v\kappa_1(\wtil{L}-\kappa_1\wtil{E})h
	=-\kappa_1df(\bm{V}_1)(=-v\kappa_1\bx_1).
	\]
	Thus we obtain the assertion for $j=1$. 
	For the case of $j=2$, one can show in a similar way. 
\end{proof}

We give definitions of ridge points and sub-parabolic points for frontals as follows. 
\begin{dfn}[cf. \cite{ifrt,tera,tera0}]
	Let $f\colon(\R^2,0)\to(\R^3,0)$ be a frontal and $0$ a pure-frontal singular point of $f$. 
	Let $\kappa_j$ $(j=1,2)$ be the principal curvatures of $f$ and $\bm{V}_j$ the corresponding principal vectors. 
	Then 
	\begin{itemize}
		\item a point $0$ is a {\it $\bm{V}_j$-ridge point} if $\bm{V}_j\kappa_j=0$ holds at $0$; 
		\item a point $0$ is a {\it $k$-th order $\bm{V}_j$-ridge point} if $\bm{V}_j\kappa_j=\cdots=\bm{V}_j^{k}\kappa_j=0$ 
		and $\bm{V}_j^{k+1}\kappa_j\neq0$ hold at $0$, where $\bm{V}_j^{m}\kappa_j$ means the $m$-th order directional derivative 
		of $\kappa_j$ in the direction $\bm{V}_j$; 
		\item a point $0$ is a {\it $\bm{V}_j$-sub-parabolic point} if $\bm{V}_j\kappa_{j+1}=0$ holds at $0$, 
		where we consider $\kappa_3$ as $\kappa_1$. 
	\end{itemize}
\end{dfn}
We shall see geometrical interpretations for these using geometric invariants in the next subsection. In the rest of this work, we use $\kappa_3=\kappa_1$ as in the above definition.

\subsection{Geometric invariants}\label{subsec:geoinv}
We recall geometric invariants of a frontal defined along a singular curve. 
Let $f\colon(\R^2,0)\to(\R^3,0)$ be a frontal, $\nu$ a unit normal vector to $f$ 
and $0$ a singular point of the first kind of $f$. 
Let $\gamma$ be a singular curve for $f$ passing through $0$. 
Then along $\gamma$, we can define the following geometric invariants: 
the {\it singular curvature} $\kappa_s$ (\cite{suy_front}), the {\it limiting normal curvature} $\kappa_\nu$ (\cite{suy_front,msuy}), 
the {\it cuspidal torsion} $\kappa_t$ (\cite{ms}), the {\it cuspidal curvature} $\kappa_c$ (\cite{msuy}), 
the {\it bias} $r_b$ (\cite{os,hs}) and the {\it secondary cuspidal curvature} $r_c$ (\cite{os,hs}).  
If $0$ is a pure-frontal singular point, then $\kappa_c$ vanishes along the singular curve $\gamma$ (\cite{msuy}). 
Moreover, a pure-frontal singular point $0$ is a $5/2$-cuspidal edge of a frontal $f$ if and only if $r_c$ does not vanish 
(\cite[Theorem 4.1]{hks} and \cite[(3.12)]{hs}). 
For precise definitions and geometrical properties, 
see \cite{suy_front,msuy,ms,os,hs,hnsuy}. 
Using functions as in \eqref{eq:fundamentals}, we 
have the following characterizations (see \cite[Lemma 2.5 and Corollary 2.6]{saji_tera}). 
\begin{lem}\label{lem:invariant}
	Let $f\colon(\R^2,0)\to(\R^3,0)$ be a frontal and $0$ a pure-frontal singular point of $f$. 
	Take an adapted coordinate system $(u,v)$ on $(\R^2,0)$. 
	Then we have the following:
	\[
	\begin{aligned}
		\kappa_\nu(u)&=\wtil{L}(u,0),\quad \kappa_t(u)=\wtil{M}(u,0),\quad  \kappa_c(u)=2\wtil{N}(u,0)(\equiv0),\\
		r_b(u)&=3\wtil{N}_v(u,0)=3\wtil{N}_1(u,0),\\
		r_c(u)&=12\left(\wtil{N}_{vv}-4\wtil{F}_v\wtil{M}-2\wtil{G}_v\wtil{N}_v\right)(u,0)\\
		&=24\left((\wtil{N}_1)_{v}-2\wtil{F}_v\wtil{M}-\wtil{G}_v\wtil{N}_1\right)(u,0),
	\end{aligned}
\]
	where $\wtil{N}=v\wtil{N}_1$. Moreover, $\wtil{E}(u,0)=\wtil{G}(u,0)=1$ and $\wtil{F}(u,0)=0$.
\end{lem}

We give the following characterizations of $\bV_j$-ridge and $\bV_{j+1}$-sub-parabolic point of a frontal 
by geometric invariants. 
\begin{prop}\label{prop:ridge}
	Let $f\colon(\R^2,0)\to(\R^3,0)$ be a frontal, $\nu$ its unit normal vector and $0$ a pure-frontal singular point. 
	Suppose that $\kappa_t\neq0$ holds at $0$. 
	Then for each $j=1,2$, a point $0$ is both a $\bm{V}_j$-ridge and $\bm{V}_{j+1}$-sub-parabolic point of $f$ if and only if $r_c=0$ at $0$.
\end{prop}
\begin{proof}
	We first show the case of $j=1$. 
	Take an adapted coordinate system $(u,v)$. 
	Then by \eqref{eq:princ-vect} and Lemma \ref{lem:invariant}, $\bm{V}_1\kappa_1$ and $\bm{V}_2\kappa_1$ can be written as 
	\[\bm{V}_1\kappa_1=(\kappa_\nu-\kappa_1)(\kappa_1)_v,\quad 
	\bm{V}_2\kappa_1=(\kappa_\nu-\kappa_2)(\kappa_1)_v\]
	at $0$. 
	We now remark that 
	\[\kappa_\nu-\kappa_j
	=\frac{1}{2}
	\left\{\left(\kappa_\nu-\frac{r_b}{3}\right)+(-1)^j\sqrt{\left(\kappa_\nu-\frac{r_b}{3}\right)^2+4\kappa_t^2}\right\}\neq0\]
	at $0$ for $j=1,2$ if $\kappa_t\neq0$ (\cite[(3.8)]{saji_tera}). 
	Thus $\kappa_j\neq\kappa_\nu$ at $0$, and hence 
	$\bm{V}_1\kappa_1=\bm{V}_2\kappa_1=0$ at $0$ is equivalent to $(\kappa_1)_v=0$ at $0$. 
	We write $\kappa_1$ as $\kappa_1=H+\sqrt{\Gamma}$, where $\Gamma=H^2-K$. 
	Then we have 
	\[
	\begin{aligned}
		(\kappa_1)_v&=H_v+\left(HH_v-\frac{K_v}{2}\right)\Gamma^{-1/2}
		=\frac{1}{\sqrt{\Gamma}}\left(H_v(\sqrt{\Gamma}+H)-\frac{K_v}{2}\right)\\
		&=\frac{1}{\sqrt{\Gamma}}\left(H_v\kappa_1-\frac{K_v}{2}\right).
	\end{aligned}
	\]
	By \cite[Lemma 4.3]{hs}, $H_v$ and $K_v$ are represented as 
	\[H_v=\frac{r_c}{48},\quad K_v=\frac{r_{\Pi}}{24}=\frac{\kappa_\nu r_c}{24}\]
	at $0$. 
	Therefore it holds that 
	\begin{equation}\label{eq:diff-kj}
		(\kappa_1)_v=\frac{r_c}{48\sqrt{\Gamma}}(\kappa_1-\kappa_\nu)
	\end{equation}
	at $0$. 
	This completes the proof for this case. 
	For the case of $j=2$, we show in the similar way by using 
	\[
	(\kappa_2)_v
	=\frac{-1}{\sqrt{\Gamma}}\left(H_v\kappa_2-\frac{K_v}{2}\right)
	=-\frac{r_c}{48\sqrt{\Gamma}}(\kappa_2-\kappa_\nu)
	\]
	at $0$. 
\end{proof}

\begin{cor}\label{cor:ridge}
	Under the same assumptions in Proposition \ref{prop:ridge}, 
	a pure-frontal singular point $0$ of a frontal $f\colon(\R^2,0)\to(\R^3,0)$ is not a $5/2$-cuspidal edge 
	if and only if $0$ is a $\bm{V}_j$-ridge and $\bm{V}_{j+1}$-sub-parabolic point of $f$ for each $j=1,2$. 
\end{cor}
\begin{proof}
	Since $f$ at $0$ is a $5/2$-cuspidal edge if and only if $r_c\neq0$, 
	we have the assertion by Proposition \ref{prop:ridge}. 
\end{proof}
This gives a geometrical interpretation for a pure-frontal singular point 
either a $5/2$-cuspidal edge or not. 
 We note that an intrinsic criterion for $5/2$-cuspidal edge is given by \cite[Corollary 4.5]{hs} 
	(see also \cite{hnsuy}).

\section{Normal congruence and its singular value sets}\label{sec:congruence}
We consider a normal congruence of a frontal $f\colon(\R^2,0)\to(\R^3,0)$. 
For cases of a regular surface and a front, see \cite{ist-congru,tera1}.

We assume that principal curvatures $\kappa_j$ $(j=1,2)$ of $f$ do not vanish on $(\R^2,0)$. 
A {\it normal congruence} $\F\colon(\R^3,0)\to(\R^3,0)$ is given by 
\begin{equation}\label{eq:F}
	\F(u,v,w)=f(u,v)+w\nu(u,v),
\end{equation}
where $\nu$ is a unit normal vector to $f$. 
Calculating the Jacobian $\det J_{\F}$ of $\F$, 
we have 
\[\det J_{\F}=\det(\F_u,\F_v,\F_w)=(1-w\kappa_1)(1-w\kappa_2)\lambda,\]
where $\kappa_j$ $(j=1,2)$ are principal curvatures of $f$ and $\lambda$ is the signed area density function of $f$. 
Thus the set of singular points of $\F$ is $S(\F)=(S(f)\times\R)\cup S_1\cup S_2$, where 
\[S_j=\{(u,v,w)\in(\R^3,0)\ |\ 1-w\kappa_j(u,v)=0\}\quad (j=1,2),\]
and hence the singular locus $\F(S(\F))$ of $\F$ is the union 
$\F(S(f)\times\R)\cup\F(S_1)\cup\F(S_2)$. 

Considering this, we define $\nr\colon(\R^2,0)\to(\R^3,0)$ and $C_j\colon(\R^2,0)\to(\R^3,0)$ $(j=1,2)$ by 
\begin{equation}\label{eq:maps}
	\nr(u,w)=\hat{\gamma}(u)+w\hat{\nu}(u),\quad 
	C_j(u,v)=f(u,v)+\rho_j(u,v)\nu(u,v),
\end{equation}
where $\hat{\nu}=\nu\circ\gamma$ and $\rho_j=1/\kappa_j$ $(j=1,2)$. 
We notice that the image of $\nr$ coincides with $\F(S(f)\times\R)$, and the image of $C_j$ coincides with $\F(S_j)$ $(j=1,2)$. 
We call $\nr$ and $C_j$ $(j=1,2)$ the {\it normal ruled surface} along $\hat{\gamma}$ 
and the ({\it normal}) {\it focal surfaces} or {\it caustics} of $f$ associated to $\kappa_j$, respectively. 
In the rest of this section, we study singularities of $\nr$. 
We shall investigate singularities and certain geometric properties of $C_j$ in the next section.

\subsection{Singularities of $\nr$}
We deal with the normal ruled surface $\nr$ along the singular locus $\hat{\gamma}$ of a frontal $f$ 
with a pure-frontal singular point. 
For a frontal $f\colon(\R^2,0)\to(\R^3,0)$, we take an adapted coordinate system $(u,v)$. 
We set $\hat{h}$ by $\hat{h}(u)=h(u,0)$, 
where $h$ is a map satisfying $f_v=vh$. 
In this case, $\{\hat{\gamma}',\hat{h},\hat{\nu}\}$ is an orthonormal frame along $\hat{\gamma}$. 
Moreover, we have the following formula (cf. \cite[Proposition 3.1]{ist_cusp} and \cite[Lemma 1.3]{fukui}):
\begin{equation}\label{eq:frenet}
	\begin{pmatrix} \hat{\gamma}' \\ \hat{h} \\ \hat{\nu} \end{pmatrix}'=
	\begin{pmatrix} 0 & \kappa_s & \kappa_\nu \\
		-\kappa_s & 0 & \kappa_t\\
		-\kappa_\nu & -\kappa_t & 0
	\end{pmatrix}
	\begin{pmatrix} \hat{\gamma}' \\ \hat{h} \\ \hat{\nu} \end{pmatrix}.
\end{equation}
We assume that $\nr$ is {\it noncylindrical}, that is, $\hat{\nu}'$ does not vanish identically. 
By \eqref{eq:frenet}, this condition is equivalent to $(\kappa_\nu,\kappa_t)\neq(0,0)$ along $\gamma$. 

\begin{lem}\label{lem:sing-nr}
	A point $(u_0,w_0)$ is a singular point of $\nr$ as in \eqref{eq:maps} 
	if and only if $\kappa_t(u_0)=0$ and $w_0=1/\kappa_\nu(u_0)$ hold. 
\end{lem}
\begin{proof}
	By \eqref{eq:frenet}, we have 
	\begin{equation}\label{eq:diff-nr}
		\nr_u=(1-w\kappa_\nu)\hat{\gamma}'-w\kappa_t\hat{h},\quad 
		\nr_w=\hat{\nu}.
	\end{equation}
	Thus it holds that 
	\[\nr_u\times\nr_w=-(1-w\kappa_\nu)\hat{h}-w\kappa_t\hat{\gamma}'.\]
	Hence we have the assertion.
\end{proof}

We next consider $\nr$ to be a developable surface or not. 
For this, we give the following characterization:
\begin{prop}\label{prop:develop}
	The normal ruled surface $\nr$ is developable  
	if and only if $\kappa_t=0$ along $\gamma$. 
\end{prop}
\begin{proof}
	Let us take an adapted coordinate system $(u,v)$. 
	Then by \eqref{eq:frenet}, we see that 
	\[\det(\hat{\gamma}',\hat{\nu},\hat{\nu}')=\kappa_t\]
	holds along the $u$-axis. 
	Hence we have the assertion (see \cite[Page 194]{docarmo}).
\end{proof}

This implies that when the singular curve of an initial frontal is a line of curvature (\cite[Proposition 3.3]{saji_tera}), 
then $\nr$ is a developable surface. 
Moreover, we have the following.
\begin{cor}
	If $\nr$ is a developable surface, then the singular locus of $\nr$ is 
	$\hat{\gamma}+\hat{\nu}/\kappa_\nu$. 
\end{cor}
\begin{proof}
	By Lemma \ref{lem:sing-nr} and Proposition \ref{prop:develop}, the conclusion follows.
\end{proof}

If $\nr$ is developable, then we can take $\hat{h}$ as a unit normal vector field. 
Thus $\nr$ is a frontal. 
Moreover, since $\nr$ is noncylindrical, $\kappa_\nu\neq0$ holds, 
and hence all singular points of $\nr$ are non-degenerate. 
By \eqref{eq:diff-nr}, $\partial_u$ can be taken as a null vector field $\eta_{\nr}$ of $\nr$. 
Thus we have the following.
\begin{prop}\label{prop:frontness1}
	Assume that the normal ruled surface $\nr$ is developable.
	Then $\nr$ is a front if and only if $\kappa_s\neq0$. 
\end{prop}
\begin{proof}
	To show this, it is sufficient to check the condition that $\eta_{\nr}\hat{h}\neq0$. 
	By \eqref{eq:frenet}, we have 
	$\eta_{\nr}\hat{h}=-\kappa_s\hat{\gamma}'$. 
	This implies the conclusion.
\end{proof}

\begin{cor}\label{cor:plane}
	If the developable normal surface $\nr$ is a frontal but not a front on $S(f)\times\R$, 
	then the image of $\nr$ is a part of a plane.
\end{cor}
\begin{proof}
	Let us take an adapted coordinate system $(u,v)$. 
	Then by Proposition \ref{prop:frontness1}, 
	if $\nr$ is a frontal but not a front, then $\hat{h}'$ vanishes identically along the $u$-axis. 
	This implies that $\hat{h}$ is a constant vector. 
	Thus we have the assertion.
\end{proof}
By this corollary, $\nr$ does not have a $5/2$-cuspidal edge. 
When $\kappa_s$ does not vanish identically, there are possibilities for $\nr$ to be a front or not. 
For the case that $\nr$ is developable, we have the following characterizations of singularities. 
\begin{thm}\label{thm:sing_NR}
	Let $\nr$ be a normal ruled surface as in \eqref{eq:maps} of a frontal $f$ and $q=(u_0,w_0)$ a singular point of $\nr$. 
	Suppose that $\nr$ is non-cylindrical and developable. 
	Then 
	\begin{enumerate}
		\item $\nr$ has a cuspidal edge at $q$ if and only if $\kappa_s(u_0)\kappa_\nu'(u_0)\neq0$;
		\item $\nr$ has a swallowtail at $q$ if and only if $\kappa_s(u_0)\neq0$, $\kappa_\nu'(u_0)=0$ and $\kappa_\nu''(u_0)\neq0$;
		\item $\nr$ has a cuspidal cross cap at $q$ if and only if $\kappa_s(u_0)=0$ and $\kappa_s'(u_0)\kappa_\nu'(u_0)\neq0$;
		\item $\nr$ has a cuspidal $S_1^+$ singularity at $q$ if and only if $\kappa_s(u_0)=\kappa_s'(u_0)=0$ and $\kappa_s''(u_0)\kappa_\nu'(u_0)\neq0$.
	\end{enumerate}
	Here a swallowtail is defined as a germ $\cal{A}$-equivalent to $(u,v)\mapsto(u,4v^3+2uv,3v^4+uv^2)$ at the origin (see Figure \ref{fig:sw}). 
	\begin{figure}[htbp]
	\centering
			\includegraphics[width=2.5cm]{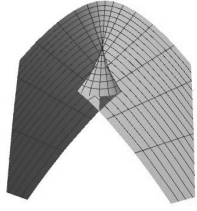}
			\caption{The image of a standard swallowtail.}
		\label{fig:sw}
	\end{figure}
\end{thm}
\begin{proof}
	By Proposition \ref{prop:frontness1}, $\nr$ at $q$ is a front if and only if $\kappa_s(u_0)\neq0$. 
	Moreover, an identifier of singularities $\Lambda$ of $\nr$ is given as $\Lambda=1-w\kappa_\nu(u)$, 
	and $\eta_{\nr}=\partial_u$ gives a null vector field of $\nr$. 
	Thus we have 
	\[\eta_{\nr}\Lambda(q)=-w_0\kappa_\nu'(u_0),\quad \eta_{\nr}\eta_{\nr}\Lambda(q)=-w_0\kappa_\nu''(u_0).\]
	Hence we have the first two assertions by the criteria for cuspidal edges and swallowtails (\cite[Corollary 2.5]{suy_ak}). 
	
	We next consider the case that $q$ is a non-front singular point of $\nr$, 
	that is, $q$ is a singular point of the first kind and $\nr$ is not a front at $q$. 
	Since $\nr$ is non-cylindrical, $\kappa_\nu\neq0$ holds, and hence the singular curve $\beta$ of $\nr$ 
	can be parametrized as $\beta(u)=(u,1/\kappa_\nu(u))$. 
	We set $\hat{\beta}(u)=\nr\circ\beta(u)$ and $\psi(u)=\det(\hat{\beta}'(u),\hat{h}(u),\eta_{\nr}\hat{h}(u))$. 
	By \eqref{eq:frenet}, we have 
	\[\psi(u)=-\frac{\kappa_s(u)\kappa_\nu'(u)}{\kappa_\nu(u)^2}.\]
	When $\nr$ is not a front at $q$, then $\psi(u_0)=0$. 
	Moreover, $\psi'(u_0)$ is calculated as 
	\[\psi'(u_0)=-\frac{\kappa_s'(u_0)\kappa_\nu'(u_0)}{\kappa_\nu(u_0)^2}.\]
	By the criterion for a cuspidal cross cap (\cite[Corollary 1.5]{fsuy}), we have the third assertion. 
	Finally, we assume that $\nr$ at $q$ is not a cuspidal cross cap. 
	Then $\psi(u_0)=\psi'(u_0)=0$ holds. 
	Since $q=(u_0,w_0)$ is of the first kind, $\kappa_\nu'(u_0)\neq0$ and $\kappa_s'(u_0)=0$ hold.
	Under these situations, we calculate $\psi''(u_0)$. 
	By a direct computation with \eqref{eq:frenet}, 
	we have 
	\[\psi''(u_0)=-\frac{\kappa_s''(u_0)\kappa_\nu'(u_0)}{\kappa_\nu(u_0)^2}{(=B)}.\]
	On the other hand, there exists a null vector field $\wtil{\eta}_{\nr}$ such that 
	\[\inner{\hat{\beta}'(u_0)}{\wtil{\eta}_{\nr}^2\nr(q)}
	=\inner{\hat{\beta}'(u_0)}{\wtil{\eta}_{\nr}^3\nr(q)}=0,\] 
	where $\wtil{\eta}_{\nr}^{k}\nr$ 
	means the $k$-time directional derivative of $\nr$ in the direction $\wtil{\eta}_{\nr}$ (\cite{os,hs,hks}). 
	We note that $\wtil{\eta}_{\nr}^2\nr\neq0$ at $q$. 
	By the similar calculations in the proof of \cite[Lemma 2.4]{saji_tera}, we have 
	\[\wtil{\eta}_{\nr}=\partial_u+\kappa_\nu'(u_0)(u-u_0)^2\partial_w.\]
	In particular, $\eta_{\nr}=\wtil{\eta}_{\nr}$ holds at $q$.
	Using this vector field, we have 
	\[\wtil{\eta}_{\nr}^2\nr=-\frac{\kappa_\nu'}{\kappa_\nu}\hat{\gamma}',\quad 
	\wtil{\eta}_{\nr}^3\nr=-\frac{\kappa_\nu''}{\kappa_\nu}\hat{\gamma}'\]
	at $q$, and hence it holds that $\wtil{\eta}_{\nr}^3\nr(q)=(\kappa_\nu''(u_0)/\kappa_\nu'(u_0))\wtil{\eta}_{\nr}^2\nr(q)$. 
	To characterize cuspidal $S_1$ singularities, we set $C=\kappa_\nu''(u_0)/\kappa_\nu'(u_0)$. 
	The cross product of $\hat{\beta}'\times\wtil{\eta}_{\nr}^2\nr$ can be written as
	\[\hat{\beta}'\times\wtil{\eta}_{\nr}^2\nr=\frac{(\kappa_\nu')^2}{\kappa_\nu^3}\hat{\nu}\times\hat{\gamma}'\]
	at $q$. 
	We calculate the fourth and fifth order directional derivatives of $\nr$ in the direction $\wtil{\eta}_{\nr}$. 
	By straightforward computations, we have 
	\[\wtil{\eta}_{\nr}^4\nr\equiv0,\quad 
	\wtil{\eta}_{\nr}^5\nr\equiv-\dfrac{4\kappa_\nu'\kappa_s''}{\kappa_\nu}\hat{h} \mod \langle{\hat{\gamma}',\hat{\nu}}\rangle_{\R}\]
	at $q$ since $\kappa_s=\kappa_s'=0$ at $u_0$. 
	Here for $a,b,x,y\in\R^3$, $a\equiv b\mod\langle{x,y}\rangle_\R$ means that 
	there exist $c_1,c_2\in\R$ such that $a-b=c_1x+c_2y$ holds. 
	Thus we obtain
	\[A=\det(\hat{\beta}',\wtil{\eta}_{\nr}^2\nr,3\wtil{\eta}_{\nr}^5\nr-10C\wtil{\eta}_{\nr}^4\nr)(q)=
	-\dfrac{12\kappa_\nu'(u_0)^3\kappa_s''(u_0)}{\kappa_\nu(u_0)^4.}\]
	Hence the product of $A$ and $B$ is 
	\[AB=\dfrac{12\kappa_\nu'(u_0)^4\kappa_s''(u_0)^2}{\kappa_\nu(u_0)^6}>0.\]
	Therefore we obtain the conclusion by the criterion \cite[Theorem 3.2]{saji_S1}.
\end{proof}
We note that a developable normal ruled surface does not admit cuspidal $S_1^-$ singularities.
This comes from the general theory \cite{ishikawa} (see also \cite{saji_S1}). 
We note that duality of singularities on flat surfaces which contain developable surfaces are studied by Honda \cite{honda}. 

We turn to our consideration to nondevelopable case. 
It is known that generic singularities of such surfaces is a {\it cross cap} (or a {\it Whitney umbrella}) (see \cite{it_ruled}), 
which is defined as a map germ $\cal{A}$-equivalent to $(u,v)\mapsto(u,uv,v^2)$ at the origin. 
Moreover, {\it $S_1^{\pm}$ singularities} (or {\it Chen Matumoto Mond $\pm$ singularities}) 
defined by map germs $\cal{A}$-equivalent to $(u,v)\mapsto(u,v^2,v(u^2\pm v^2))$ at the origin are known as 
the generic singularities of one-parameter families of $2$-manifolds in $\R^3$ (\cite{cm,mond}). 
\begin{figure}[htbp]
	\centering
		\begin{tabular}{c}
			
			\begin{minipage}{0.3\hsize}
				\begin{center}
					\includegraphics[clip, width=3cm]{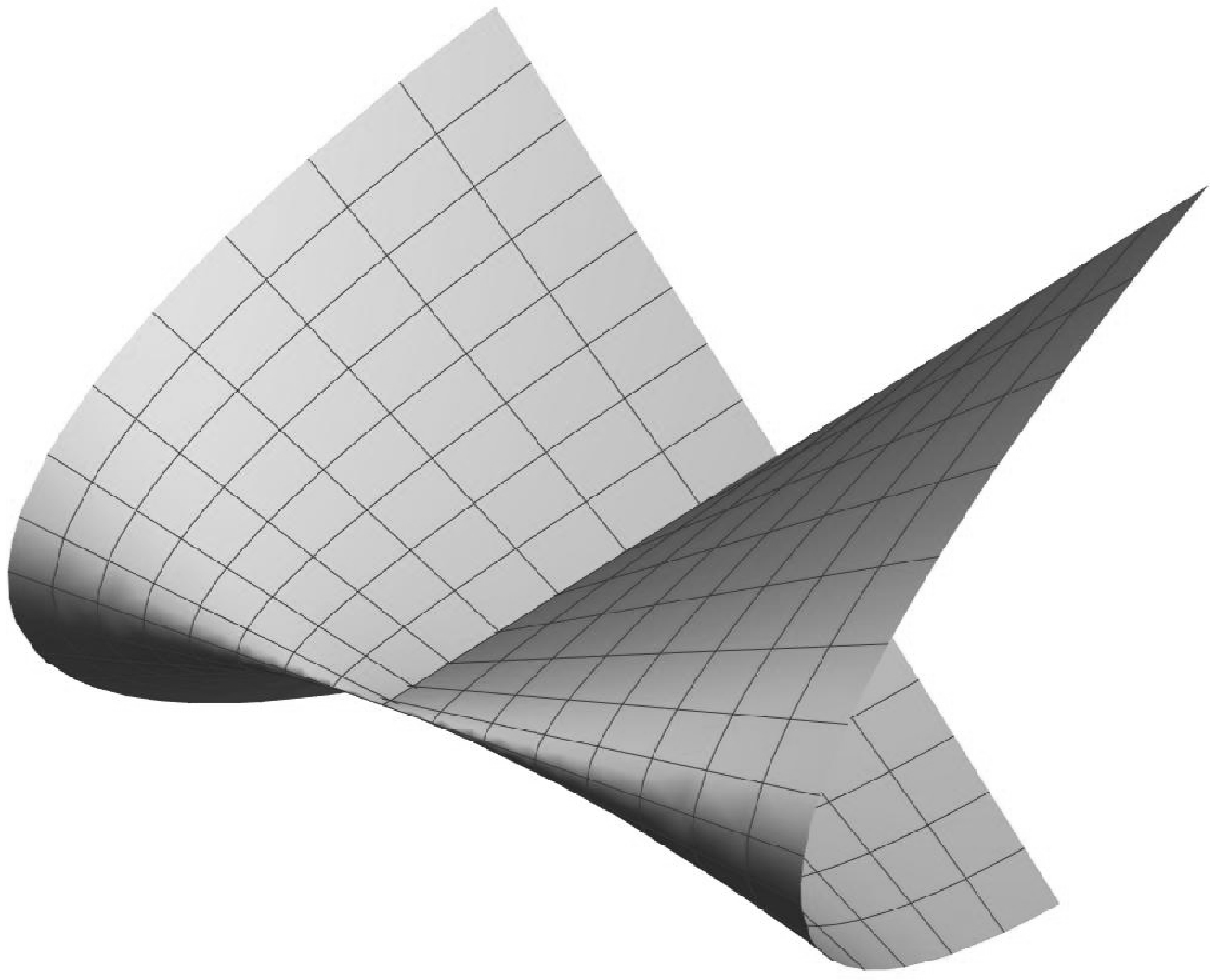}
				\end{center}
			\end{minipage}
			
			\begin{minipage}{0.3\hsize}
				\begin{center}
					\includegraphics[clip, width=3cm]{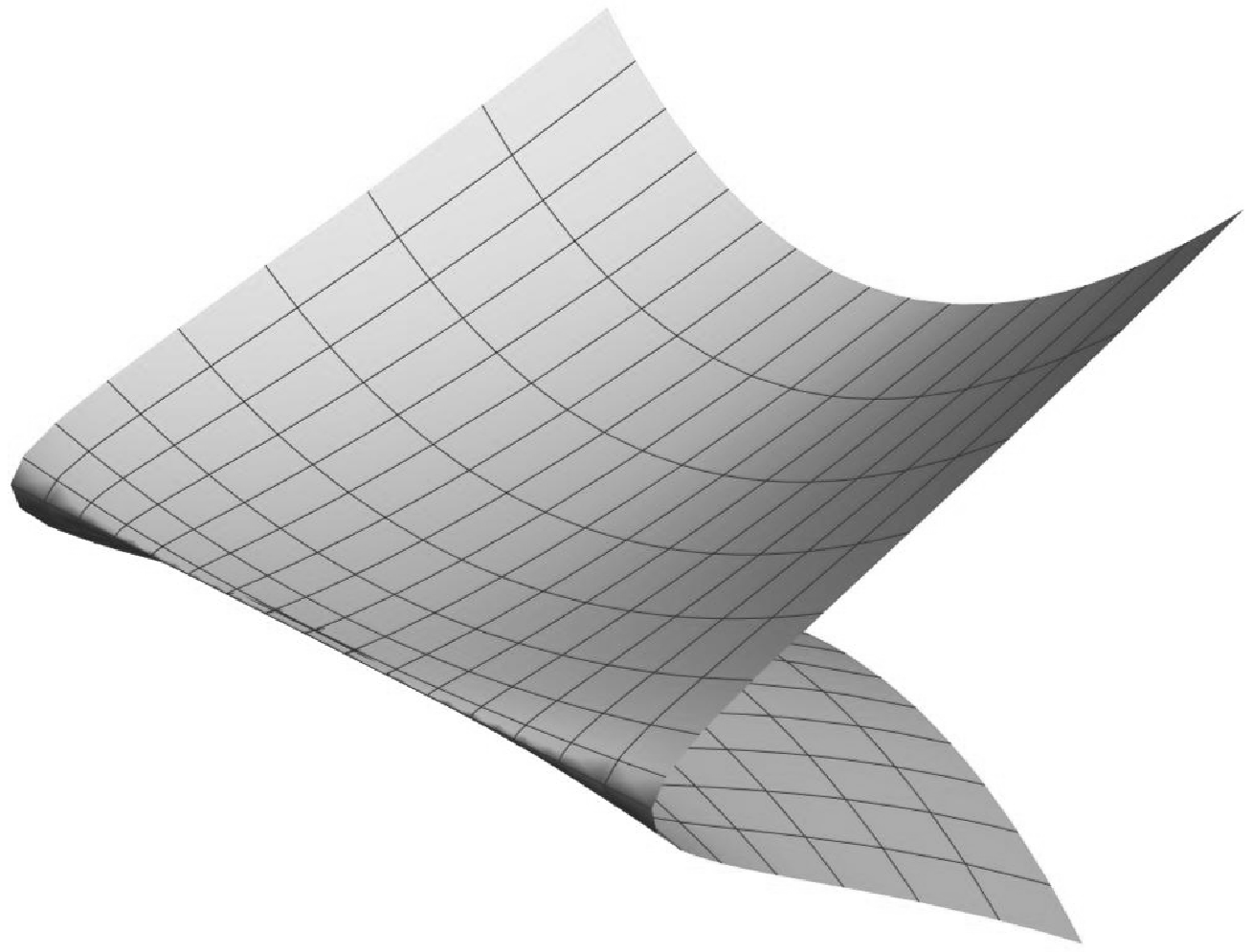}
				\end{center}
			\end{minipage}
			
			\begin{minipage}{0.3\hsize}
				\begin{center}
					\includegraphics[clip, width=3cm]{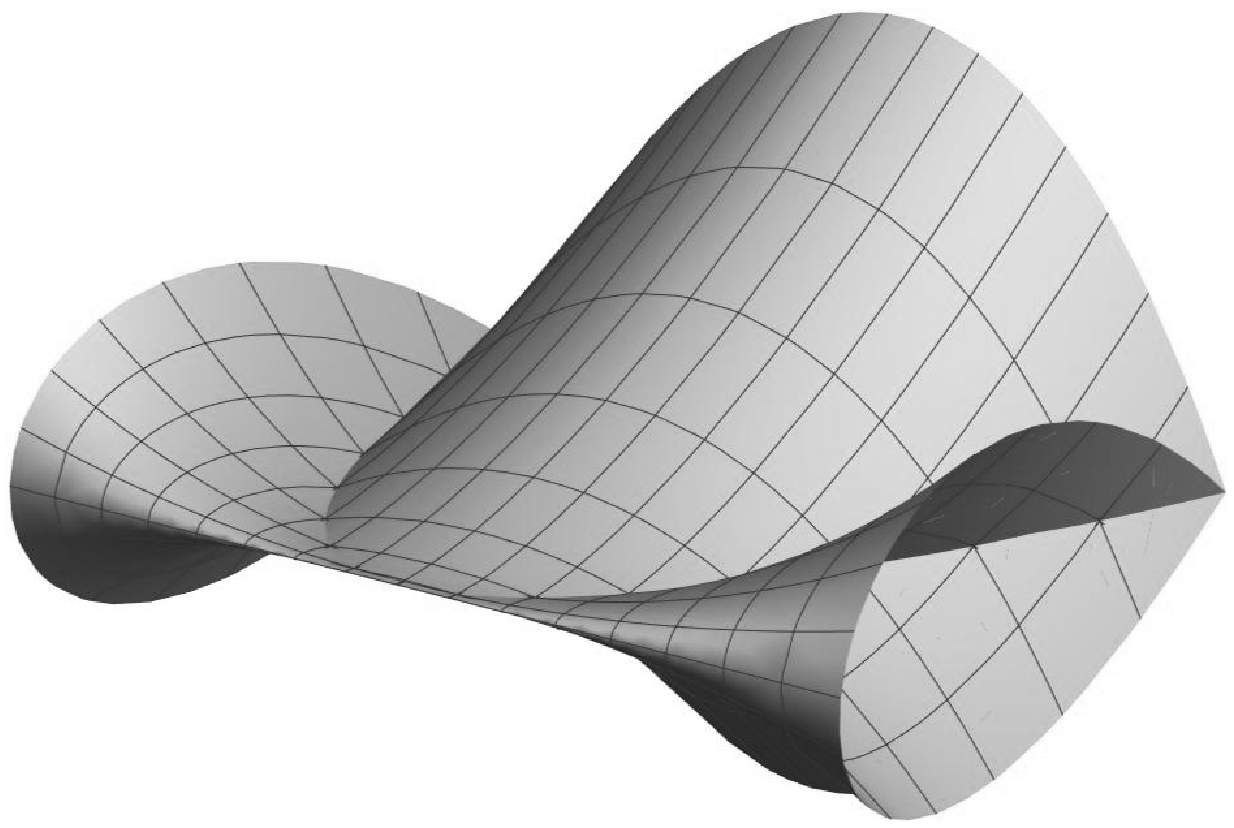}
				\end{center}
			\end{minipage}
		\end{tabular}
		\caption{The images of a cross cap (left), an $S_1^+$ singularity (middle) and an $S_1^{-}$ singularity (right).}
		\label{fig:singular2}
\end{figure}
We give characterization of these singularities on a noncylindrical nondevelopable normal ruled surface 
in terms of geometric invariants of the initial frontal surface by using the criteria obtained in \cite{saji_S1}. 

Let $f\colon(\R^2,0)\to(\R^3,0)$ be a frontal and $0$ a pure-frontal singular point of $f$. 
Let $\nr$ be the normal ruled surface on $(S(f)\times\R;u,w)$ as in \eqref{eq:maps}{.}
Assume that $\nr$ is noncylindrical and nondevelopable. 
Then the set of singular points of $\nr$ is 
\begin{equation}\label{eq:singset_nr}
	S(\nr)=\{(u_0,w_0)\in S(f)\times\R\ |\ \kappa_t(u_0)=0,\ 1-w_0\kappa_\nu(u_0)=0\}
\end{equation}
by Lemma \ref{lem:sing-nr}. 
\begin{thm}\label{thm:nondev-NR}
	Let $f\colon(\R^2,0)\to(\R^3,0)$ be a frontal with pure-frontal singular point $0$. 
	Suppose that a normal ruled surface $\nr$ as in \eqref{eq:maps} of $f$ is noncylindrical and nondevelopable. 
	Then we have the following.
	\begin{enumerate}
		\item $\nr$ has a cross cap at $(u_0,w_0)\in S(\nr)$ if and only if $\kappa_t'(u_0)\neq0$.
		\item $\nr$ has a $S_1^-$ singularity at $(u_0,w_0)\in S(\nr)$ if and only if $\kappa_t'(u_0)=0$ and 
	$\kappa_t''(u_0)(2\kappa_s(u_0)\kappa_\nu'(u_0)+\kappa_t''(u_0))<0$.
		\item $\nr$ has a $S_1^+$ singularity at $(u_0,w_0)\in S(\nr)$ if and only if $\kappa_t'(u_0)=0$, $\kappa_\nu'(u_0)\neq0$ and 
		$\kappa_t''(u_0)(2\kappa_s(u_0)\kappa_\nu'(u_0)+\kappa_t''(u_0))>0$.
	\end{enumerate}
\end{thm}
\begin{proof}
	Let us take an adapted coordinate system $(u,v)$ for $f$. 
	Then $S(f)=\{v=0\}$ holds. 
	Under this setting, we show the assertions. 
	Putting $\xi_{\nr}=\partial_w$ and $\eta_{\nr}=\partial_u$ on $S(f)\times\R$, 
	the pair $(\xi_{\nr},\eta_{\nr})$ satisfies that $d\nr(\eta_{\nr})(=\eta_{\nr}\nr)=0$ 
	and $\xi_{\nr}$, $\eta_{\nr}$ are linearly independent at $q\in S(\nr)$. 
	Using this pair $(\xi_{\nr},\eta_{\nr})$, we define a function $\phi\colon S(f)\times\R\to\R$ by 
	$\phi=\det(\xi_{\nr}\nr,\eta_{\nr}\nr,\eta_{\nr}\eta_{\nr}\nr)=\det(\nr_w,\nr_u,\nr_{uu})$. 
	Since $\nr_w, \nr_u$ and $\nr_{uu}$ can be calculated as
	\[
	\nr_w=\hat\nu, \ \ \ \nr_u=(1-w\kappa_\nu)\hat\gamma'w \kappa_t \hat{h}\
	\]
	\[
	\nr_{uu}=w(\kappa_s\kappa_t-\kappa_\nu')\hat\gamma+(\kappa_s(1-w\kappa_\nu)-w \kappa_t')\hat{h}+*\hat\nu
	\]
	by \eqref{eq:diff-nr}, where $*$ is a some function of $u$ and $w$, we have
\[\phi=w^2\left(\kappa_s\left(\kappa_\nu^2+\kappa_t^2\right)+
\kappa_\nu\kappa_t'-\kappa_\nu'\kappa_t\right)-w\left(2\kappa_s\kappa_\nu+\kappa_t'\right)+\kappa_s\]
	In this case, $\nr$ at $q=(u_0,w_0)$ is a cross cap if and only if $\phi(q)=0$ and $\xi_{\nr}\phi(q)=\phi_w(q)\neq0$ 
	(see \cite{whitney} and \cite[Remark 2.3]{saji_S1}). 
	Thus we have the first assertion. 
	
	We next consider the case that $\nr$ does not have a cross cap at $q$, namely $\kappa_t'(u_0)=0$. 
	The first order differentials of $\phi$ satisfy $\phi_u(q)=0$ and $\phi_w(q)=\kappa_t'(u_0)$. 
	Thus $\phi$ has a critical point at $q$. 
	We consider the Hessian of $\phi$ at $q$. 
	By straightforward calculations, we have 
	\[\phi_{uu}=\frac{\kappa_\nu'}{\kappa_\nu^2}(\kappa_t''+2\kappa_s\kappa_\nu'),\quad 
	\phi_{uw}=\kappa_t''+2\kappa_s\kappa_\nu',\quad 
	\phi_{ww}=2\kappa_s\kappa_\nu^2\]
	at $q$. 
	Therefore the Hessian $\phi_{uu}\phi_{ww}-\phi_{uw}^2$ at $q$ is 
\[\phi_{uu}(q)\phi_{ww}(q)-\phi_{uw}(q)^2=-\kappa_t''(u_0)(2\kappa_s(u_0)\kappa_\nu'(u_0)+\kappa_t''(u_0)).\]
	Moreover, when $\kappa_t'(u_0)=0$, then $\nr_{uu}=-(\kappa_\nu'(u_0)/\kappa_\nu(u_0))\hat{\gamma}'(u_0)$. 
	Thus $\nr_w(q)$ and $\nr_{uu}(q)$ are linearly independent if and only if $\kappa_\nu'(u_0)\neq0$, 
	and hence we have the assertion for $S_1$ singularities by the criteria for these singularities \cite[Theorem 2.2]{saji_S1}. 
\end{proof}

 For characterizations of singularities of $\nr$, we only need a singular point to be of the first kind. 
Thus same results hold for $\nr$ of a frontal with such a singular point.

\section{Focal surfaces of frontal surfaces with pure-frontal singular points}\label{sec:focal}
We investigate singularities and certain geometric properties of $C_j$ 
as in \eqref{eq:maps} of a frontal surface with pure-frontal singularities.

\subsection{Singularities of $C_j$}
We here consider singularities of focal surface $C_j$ of a frontal surface $f$ with pure-frontal singular point.
Let $\bm{V}_j$ be the principal vector relative to $\kappa_j$. 
If $\kappa_t\neq0$, then it is known that there exists a never vanishing smooth map $\bx_j\colon(\R^2,0)\to\R^3$ such that 
$df(\bm{V}_j)={\Lambda}\bx_j$ (\cite[Corollary 3.4]{saji_tera}), where $\Lambda$ is the identifier of singularities. 

\begin{lem}
	The map $\bx_j$ $(j=1,2)$ is normal to $C_j$ at any $q\in(\R^2,0)$.
\end{lem}
\begin{proof}
	Let us take an adapted coordinate system $(u,v)$. 
	Then $\bm{V}_j$ can be written as 
	\[\bm{V}_j=(-v(\wtil{M}-\kappa_j\wtil{F}),\wtil{L}-\kappa_j\wtil{E}).\]
	Thus $\bx_j$ is given by 
	\begin{equation}\label{eq:vec-x}
		\bx_j=-(\wtil{M}-\kappa_j\wtil{F})f_u+(\wtil{L}-\kappa_j\wtil{E})h
	\end{equation}
	(cf. \eqref{eq:dfv}). 
	We note that $\bx_j$ is perpendicular to $\nu$. 
	By direct calculation, we have 
	\[(C_j)_u\equiv f_u+\rho_j\nu_u,\quad (C_j)_v\equiv v(h+\rho_j\nu_1) \mod\ \langle{\nu}\rangle_{\mathcal{E}},\]
	where $\nu_v=v\nu_1$, $\mathcal{E}$ is the ring of $C^\infty$ function germs on $(\R^2,0)$, 
	and $\alpha\equiv\beta\mod\langle{\nu}\rangle_{\mathcal{E}}$ means that 
	there exists $a\in\mathcal{E}$ such that $\alpha-\beta=a\nu$. 
	Thus we have 
	\[
	\begin{aligned}
		\inner{(C_j)_u}{\bx_j}&=(\wtil{M}-\kappa_j\wtil{F})(-\wtil{E}+\rho_j\wtil{L})+(\wtil{L}-\kappa_j\wtil{E})(\wtil{F}-\rho_j\wtil{M})=0,\\
		\inner{(C_j)_v}{\bx_j}&=v(\wtil{M}-\kappa_j\wtil{F})(-\wtil{F}+\rho_j\wtil{M})+v(\wtil{L}-\kappa_j\wtil{E})(\wtil{G}-\rho_j\wtil{N}_1)\\
		&=-v(\wtil{E}\wtil{G}-\wtil{F}^2)(\kappa_j-2H+K\rho_j)=0,
	\end{aligned}
\]
	where $K$ and $H$ are the Gaussian and mean curvatures (see \eqref{eq:KH}). 
	This completes the proof.
\end{proof}
By this lemma, the map $\e_j=\bx_j/|\bx_j|$ gives an unit normal vector to $C_j$ $(j=1,2)$. 
Thus $C_j$ is a frontal surface when $\kappa_t\neq0$. 
\begin{prop}\label{prop:sc}
	Let $C_j$ be a focal surface of a frontal $f\colon(\R^2,0)\to(\R^3,0)$ associated to $\kappa_j$. 
	Then the set of singular points $S(C_j)$ of $C_j$ coincides with the zero set of $\bm{V}_j\rho_j$, 
	where $\bm{V}_j$ is the principal vector relative to $\kappa_j$ and $\rho_j=1/\kappa_j$. 
\end{prop}
\begin{proof}
	We give a proof for $j=1$. 
	Take an adapted coordinate system $(u,v)$ on $(\R^2,0)$. 
	Then the map $\e_1=\bx_1/|\bx_1|$ is a unit normal vector to $C_1$, 
	where $\bx_1$ is given by \eqref{eq:vec-x} satisfying $df(\bm{V}_1)=v\bx_1$. 
	We then calculate $\det((C_1)_u,(C_1)_v,\e_1)$. 
	By direct calculations, one can see 
	\[
	\begin{aligned}
		(C_1)_u&=f_u+\rho_1\nu_u+(\rho_1)_u\nu=(1+\rho_1X_1)f_u+\rho_1X_2h+(\rho_1)_u\nu\\
		&=A_1f_u+B_1h+C_1\nu,\\
		(C_1)_v&=f_v+\rho_1\nu_v+(\rho_1)_v\nu=v\rho_1Y_1f_u+v(1+\rho_1Y_2)h+(\rho_1)_v\nu\\
		&=A_2f_u+B_2h+C_2\nu,
	\end{aligned}
\]
	where we set $\nu_u=X_1f_u+X_2h$ and $\nu_v=vY_1f_u+vY_2h$ (see Lemma \ref{lem:weingarten}). 
	Thus we have 
	\small
	\[(C_1)_u\times(C_1)_v=(A_1B_2-A_2B_1)f_u\times h+(A_1C_2-A_2C_1)f_u\times\nu+(B_1C_2-B_2C_1)h\times\nu.\]
	\normalsize
	We remark that $f_u\times h$ is parallel to $\nu$. 
	Hence $\det((C_1)_u,(C_1)_v,\e_1)$ can be calculated as 
	\small
	\[
	\begin{aligned}
		&\det((C_1)_u,(C_1)_v,\e_1)\\
		&=-\left((\wtil{L}-\kappa_1\wtil{E})(A_1C_2-A_2C_1)+(\wtil{M}-\kappa_1\wtil{F})(B_1C_2-B_2C_1)\right)\frac{\det(f_u,h,\nu)}{|\bx_1|}.
	\end{aligned}
	\]
	\normalsize
	By Lemma \ref{lem:weingarten}, we have 
	\[
	\begin{aligned}
	&\left((\wtil{L}-\kappa_1\wtil{E})(A_1C_2-A_2C_1)+(\wtil{M}-\kappa_1\wtil{F})(B_1C_2-B_2C_1)\right)\\
	&=(1-\rho_1\kappa_2)\left(v(\wtil{M}-\kappa_1\wtil{F})(\rho_1)_u-(\wtil{L}-\kappa_1\wtil{E})(\rho_1)_v\right)
	=-(1-\rho_1\kappa_2)\bm{V}_1\rho_1.
	\end{aligned}
	\]
	Since $\kappa_1\neq\kappa_2$, one can notice that $1-\rho_1\kappa_2=\rho_1(\kappa_1-\kappa_2)\neq0$ holds.
	Thus the zero set of $\det((C_1)_u,(C_1)_v,\e_1)$ coincides with the zero set of $\bm{V}_1\rho_1$. 
	This implies that $S(C_1)=(\bm{V}_1\rho_1)^{-1}(0)$. 
	For the case of $j=2$, one can prove similarly. 
\end{proof}

This result corresponds to the case of regular surfaces \cite{porteous}. 

By the above proposition, a pure-frontal singular point $0$ of a frontal $f\colon(\R^2,0)\to(\R^3,0)$ 
is also a singular point of $C_j$ if and only if $\bm{V}_j\rho_j=0$ at $0$. 
We characterize this condition in terms of geometric invariants. 
\begin{prop}\label{prop:sing_mean}
	Let $f\colon(\R^2,0)\to(\R^3,0)$ be a frontal and $0$ a pure-frontal singular point of $f$. 
	Suppose that $\kappa_t\neq0$ at $0$. 
	Then the point $0$ is also a singular point of $C_j$ $(j=1,2)$ if and only if $r_c=0$ holds at $0$.
\end{prop} 
\begin{proof}
	Take an adapted coordinate system $(u,v)$ on $(\R^2,0)$. 
	Then we have 
	$\bm{V}_j\rho_j=(\wtil{L}-\kappa_j\wtil{E})(\rho_j)_v$ at $0$. 
	We note that 
	$\wtil{L}-\kappa_j\wtil{E}=\kappa_\nu-\kappa_j\neq0$ 
	at $0$ for $j=1,2$ if $\kappa_t\neq0$. 
	Thus $\bm{V}_j\rho_j=0$ at $0$ if and only if $(\rho_j)_v=0$ at $0$ when $\kappa_t\neq0$. 
	This is equivalent to $(\kappa_j)_v=0$ at $0$. 
	By Proposition \ref{prop:ridge}, we have the assertion.
\end{proof}
This proposition means that if $0$ is a $\bm{V}_j$-ridge point (and also a $\bm{V}_{j+1}$-sub-parabolic point) of $f$, 
then $0$ is a singular point of $C_j$ (see Proposition \ref{prop:ridge}).
By Corollary \ref{cor:ridge} and Proposition \ref{prop:sing_mean}, 
the following assertion follows immediately. 
\begin{cor}\label{cor:sing_sing}
	If a frontal $f\colon(\R^2,0)\to(\R^3,0)$ has a pure-frontal singular point other than a $5/2$-cuspidal edge at $0$ and $\kappa_t\neq0$, 
	then $0$ is a singular point of both $C_1$ and $C_2$. 
\end{cor}
By this corollary, when a frontal $f$ has a fold singular point at $0$, 
then $C_j$ has a singularity at $0$ if $\kappa_t\neq0$. 
We remark that {\it maxfaces} introduced by Umehara and Yamada \cite{uy-maxface}, 
that is, spacelike zero mean curvature surfaces in the Minkowski $3$-space $\R^3_1$ (or $L^3$), 
with fold singularities (cf. \cite{fkkrsuyy}) 
are examples of such surfaces when we consider the ambient space as the Euclidean $3$-space $\R^3$ (cf. \cite{mnt}). 
Moreover, the {\it $7/2$-cuspidal cross cap} (see \cite[Example 14.10]{porteous}), which is defined as a map germ $\cal{A}$-equivalent to 
$(u,v)\mapsto(u,v^2,uv^5)$ at the origin, with non vanishing $\kappa_t$ is other typical example (see Figure \ref{fig:27ccr}). 
\begin{figure}[htbp]
	\centering
		\includegraphics[width=3cm]{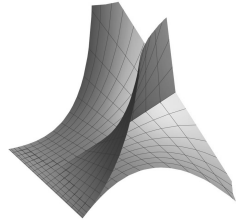}
		\caption{The image of a $7/2$-cuspidal cross cap with non vanishing $\kappa_t$.}
	\label{fig:27ccr}
\end{figure}

\begin{lem}\label{lemma:nullvectorCj}
	Let $f\colon(\R^2,0)\to(\R^3,0)$ be a frontal and $0$ a pure-frontal singular point of $f$. 
	Suppose that $\kappa_t\neq0$ and $r_c=0$ hold at $0$. 
	Then $C_j$ $(j=1,2)$ satisfies $\rank dC_j(0)=1$. 
	Moreover, $\bm{V}_j$ is a null vector field of $C_j$. 
\end{lem}
\begin{proof}
	Take an adapted coordinate system $(u,v)$. 
	Then we see that 
	\[(C_j)_u=f_u+\rho_j\nu_u+(\rho_j)_u\nu,\quad 
	(C_j)_v=v(h+\rho_j\nu_1)+(\rho_j)_v\nu\]
	hold. 
	By the assumption, $(\rho_j)_v=0$ holds at $0$, and hence $(C_j)_v=0$ holds at $0$. 
	On the other hand, by Lemmas \ref{lem:weingarten} and \ref{lem:invariant}, we have 
	\[(C_j)_u=(1-\rho_j\kappa_\nu)f_u-\rho_j\kappa_th+(\rho_j)_u\nu\neq0\]
	at $0$. 
	Thus $\rank dC_j=1$ holds at $0$. 
	By this property, there exists a null vector field $\eta^{C_j}$ $(j=1,2)$ for $C_j$ around $0$. 
	On the other hand, by Lemma \ref{lem:rod}, we see that 
	\[dC_j(\bm{V}_j)=df(\bm{V}_j)+\rho_jd\nu(\bm{V}_j)+(\bm{V}_j\rho_j)\nu=(\bm{V}_j\rho_j)\nu\]
	holds. 
	Since $S(C_j)=(\bm{V}_j\rho_j)^{-1}(0)$ holds by Proposition \ref{prop:sc}, 
	it follows that $dC_j(\bm{V}_j)=0$ holds on $S(C_j)$. 
	Therefore one can take $\eta^{C_j}$ as $\eta^{C_j}=\bV_j$.
	This shows the conclusion.
\end{proof}
This property also holds for cases of regular surfaces without umbilic points 
and fronts with non-degenerate singular points (\cite{ifrt,porteous,tera1}).

\begin{prop}\label{prop:frontalness}
	If a point $0$ is a pure-frontal singular point of a frontal $f\colon(\R^2,0)\to(\R^3,0)$ 
	with $\kappa_t\neq0$ and also a singular point of $C_j$ $(j=1$ or $2)$, 
	then $C_j$ is a frontal but not a front at $0$.
\end{prop}
\begin{proof}
	By Lemma \ref{lemma:nullvectorCj}, it is sufficient to show $d\e_j(\bV_j)=0$ at $0$. 
	Let us take an adapted coordinate system $(u,v)$. 
	We then deal with the case of $j=1$. 
	Suppose that $\kappa_t\neq0$ and $0$ is a singular point of $C_1$. 
	By a direct calculation, we have 
	$d\e_1(\bm{V}_1)=(\wtil{L}-\kappa_1\wtil{E})(\e_1)_v$ at $0$. 
	Since $\wtil{L}-\kappa_1\wtil{E}\neq0$, we consider $(\e_1)_v$ at $0$. 
	Since $\e_1=\bx_1/|\bx_1|$, it follows that $(\e_1)_v=(|\bx_1|^{-3})((\bx_1)_v|\bx_1|^2-\bx_1\inner{\bx_1}{(\bx_1)_v})$. 
	Thus we evaluate $(\bx_1)_v|\bx_1|^2-\bx_1\inner{\bx_1}{(\bx_1)_v}$ at $0$. 
	Since $\bx_1$ is given by 
	$\bx_1=-\kappa_tf_u+(\kappa_\nu-\kappa_1)h$
	at $0$, and hence $|\bx_1|^2=\kappa_t^2+(\kappa_\nu-\kappa_1)^2$ at $0$. 
	Further, $(\bx_1)_v$ can be calculated as 
	\[(\bx_1)_v=-(\wtil{M}_v-\kappa_1\wtil{F}_v)f_u+(\wtil{L}_v-(\kappa_1)_v)h+(\wtil{L}-\kappa_1)h_v\]
	at $0$ since $\wtil{F}(0)=\wtil{E}_v(0)=0$ and $\wtil{E}(0)=1$. 
	We note that $h_v=\wtil{F}_vf_u+(\wtil{G}_v/2)h$ holds at $0$ (\cite[(2.9)]{saji_tera}). 
	Moreover, since $f_{uv}=\nu_v=0$ at $0$, we see that $\wtil{L}_v=0$ 
	and $\wtil{M}_v=\wtil{F}_v\wtil{L}+\wtil{G}_v\wtil{M}/2$ at $0$. 
	Thus $(\bx_1)_v$ can be written as 
	\begin{equation}\label{eq:x1_v}
		(\bx_1)_v=-\frac{\kappa_t\wtil{G}_v}{2}f_u
		-(\kappa_1-\kappa_\nu)\left(\frac{r_c}{48\sqrt{\Gamma}}+\frac{\wtil{G}_v}{2}\right)h
	\end{equation}
	at $0$ by \eqref{eq:diff-kj}. 
	Therefore we have 
	\[\inner{\bx_1}{(\bx_1)_v}=\frac{\kappa_t^2\wtil{G}_v}{2}
	+(\kappa_1-\kappa_\nu)^2\left(\frac{r_c}{48\sqrt{\Gamma}}+\frac{\wtil{G}_v}{2}\right)\]
	at $0$. 
	Since $0$ is a singular point of $C_j$, $r_c=0$ holds at $0$. 
	Thus we have 
	\[(\bx_1)_v=-\frac{\wtil{G}_v}{2}(\kappa_tf_u+(\kappa_1-\kappa_\nu)h)\]
	at $0$ by \eqref{eq:x1_v}. 
	Thus we have 
	\[(\bx_1)_v|\bx_1|^2=-\frac{(\kappa_t^2+(\kappa_1-\kappa_\nu)^2)\wtil{G}_v}{2}(\kappa_tf_u+(\kappa_1-\kappa_\nu)h)
	=\inner{\bx_1}{(\bx_1)_v}\bx_1\]
	at $0$. 
	This implies that $(\e_1)_v=0$ holds at $0$. 
	Therefore we get the conclusion for $j=1$. 
	For the case of $j=2$, one can show in a similar way.
\end{proof}

We seek conditions that $0$ is a singular point of the second kind for $C_j$.
\begin{lem}
	Let $f$ be a frontal and $0$ a pure-frontal singular point of $f$ with $\kappa_t(0)\neq0$. 
	Take an adapted coordinate system $(u,v)$ 
	and suppose that $0$ is a $\bV_j$-ridge point $(j=1,2)$ of $f$, that is, $r_c=0$ at $0$. 
	Then $0$ is an at least second order $\bV_j$-ridge point of $f$ if and only if 
	\[(\kappa_\nu-\kappa_j)(\kappa_j)_{vv}-\kappa_t (\kappa_j)_u=0\] 
	or, equivalently, 
	\[\kappa_t(\kappa_j)_{vv}-\left(\frac{r_b}{3}-\kappa_j\right)(\kappa_j)_u=0\]
	at $0$.
\end{lem}
\begin{proof}
	Since $0$ is a ridge point of $f$, $(\kappa_j)_v(0)=0$ holds by Proposition \ref{prop:sc}. 
	Thus we have
	\begin{equation}\label{eq:VVk}
		\bV_j\bV_j\kappa_j=(\wtil{L}-\kappa_j\wtil{E})((\wtil{L}-\kappa_j\wtil{E})(\kappa_j)_{vv}-(\wtil{M}-\kappa_j\wtil{F}) (\kappa_j)_u)
	\end{equation}
	at the origin. Since $\kappa_t\neq0$, we have $\wtil{L}-\kappa_j\wtil{E}\neq0$. 
	This implies that $0$ is an at least second order $\bV_j$-ridge point of $f$ if and only if
	\begin{equation}\label{eq:ridgeordem2}
		(\wtil{L}-\kappa_j\wtil{E})(\kappa_j)_{vv}-(\wtil{M}-\kappa_j\wtil{F}) (\kappa_j)_u=0.
	\end{equation}
	By Lemma \ref{lem:invariant}, we have the first expression. 
	Moreover, using the formula $\kappa_j^2-2H\kappa_j+K=0$, 
	we note that the vector field 
	$\bar{\bV}_j=(-v(N_1-\kappa_j\wtil{G}),\wtil{M}-\kappa_j\wtil{F})$ 
	is also a principal vector field which does not vanish at $0$, and we have 
	\begin{equation}\label{eq:VVbark}
		\bar{\bV}_j\bar{\bV}_j\kappa_j
		=(\wtil{M}-\kappa_j\wtil{G})\left((\wtil{M}-\kappa_j\wtil{F})(\kappa_j)_{vv}-(N_1-\kappa_j\wtil{G})(\kappa_j)_u\right)=0
	\end{equation}
	at $0$. 
	By Lemma \ref{lem:invariant} again, we have the second expression. 
	Since being a ridge point does not depend on the principal vector field, we get the conclusion.
\end{proof}

\begin{lem}\label{lem:non-deg}
	Let $f$ be a frontal and $0$ a pure-frontal singular point of $f$ with $\kappa_t(0)\neq0$. 
	Suppose that $0$ is a $\bV_j$-ridge point $(j=1,2)$ of $f$. 
	Then $0$ is a non-degenerated singular point of $C_j$ $(j=1,2)$ if and only if 
	$r_c'(0)\neq0$ or $0$ is a first order $\bV_j$-ridge point.
\end{lem}
\begin{proof}
	We show the case of $j=1$. The case $j=2$ is similar.
	Take an adapted coordinate system $(u,v)$. 
	Let $\lambda^{C_1}$ be the signed area density function of $C_1$. 
	Since $0$ is a $\bV_1$-ridge point, 
	{$(\bV_1\rho_1)_u$ at $0$ can be calculated as 
		\begin{equation}\label{eq:lambdaCu}
			(\bV_1\rho_1)_u=(\wtil{L}-\kappa_1\wtil{E})\rho_{uv}=-\frac{(\kappa_\nu-\kappa_1)}{\kappa_1^2}(\kappa_1)_{uv}.
		\end{equation}
		By \eqref{eq:diff-kj} (see also \cite[(3.2)]{saji_tera}), we have
		\begin{equation}\label{eq:kuv}
			(\kappa_1)_{uv}=\frac{r_c'(\kappa_1-\kappa_\nu)}{48\sqrt{\Gamma}}
			\left(=\frac{r_c'(\kappa_1-\kappa_\nu)}{8\sqrt{(r_b-3\kappa_\nu)^2+36\kappa_t^2}}\right)
		\end{equation}
		at the origin, and hence it holds that 
		\[
		(\bV_1\rho_1)_u=\frac{r_c'(\kappa_1-\kappa_\nu)^2}{48\kappa_1^2\sqrt{\Gamma}}
		\left(=\frac{r_c'(\kappa_1-\kappa_\nu)^2}{8{\kappa_1^2}\sqrt{(r_b-3\kappa_\nu)^2+36\kappa_t^2}}\right)
		\]
		at the origin. 
		Moreover, setting $\bV=V_{11}\partial_u+V_{12}\partial_v$, $\bV_1(\bV_1\rho_1)$ at $0$ is computed as 
		\begin{equation}\label{eq:lambdaCv}
			\bV_1(\bV_1\rho_1)=V_{12}(\bV_1\rho_1)_v
		\end{equation}
		since $V_{11}=0$ at the origin. 
		On the other hand, it follows that 
		\[
		\bV_1(\bV_1\rho_1)=-\bV_1\left(\frac{\bV_1\kappa_1}{\kappa_1^2}\right)
		=-\frac{\bV_1(\bV_1\kappa_1)\kappa_1^2-2\kappa_1(\bV_1\kappa_1)^2}{\kappa_1^4}
		=-\frac{\bV_1(\bV_1\kappa_1)}{\kappa_1^2}
		\]
		holds at the origin because $\bV_1\kappa_1=0$ at $0$. 
		Thus $(\bV_1\rho_1)_v\neq0$ at $0$ if and only if $\bV_1(\bV_1\kappa_1)\neq0$ at $0${,}
		since $V_{12}=\kappa_\nu-\kappa_1\neq0$ at $0$. 
		Therefore we have $((\bV_1\rho_1)_u,(\bV_1\rho_1)_v)\neq(0,0)$ 
		if and only if $r_c'(0)\neq0$ or the origin is a first order $\bV_1$-ridge point of $f$. }
	%
	%
\end{proof}

\begin{prop}\label{prop:sk}
	Let $f\colon(\R^2,0)\to(\R^3,0)$ be a frontal and $0$ a pure-frontal singularity of $f$ with $\kappa_t(0)\neq0$. 
	Suppose that $0$ is a non-degenerated singular point of $C_j$ $(j=1,2)$. 
	Then $0$ is an at least second order $\bV_j$-ridge point of $f$ if and only if $0$ is 
	a singular point of the second kind of $C_i$.
\end{prop}
\begin{proof}
	{We show the case $j=1$. 
		Take an adapted coordinate system $(u,v)$. 
		Suppose $0$ is an at least second order $\bV_1$-ridge point of $f$, 
		that is, $\bV_1\kappa_1=\bV_1\bV_1\kappa_1=0$ holds at $0$. 
		Thus, by \eqref{eq:lambdaCv}, 
		we have $(\bV_1\rho_1)_v(0)=0$. 
		Since $0$ is a non-degenerated singularity of $C_1$, we have $(\kappa_1)_{uv}(0)\neq0$ by \eqref{eq:lambdaCu}.
		By the identity
		\[
		(\bV_1\rho_1)_u=-(\kappa_\nu-\kappa_1)\dfrac{(\kappa_1)_{uv}}{\kappa_1^2}\neq0
		\]
		at $0$ and the implicit function theorem, 
		there exists a regular curve $(g_1(t),t)$ such that $\bV_1\rho_1(g_1(t),t)=0$. 
		We also have 
		\[g_1'(t)=-\frac{(\bV_1\rho)_v}{(\bV_1\rho)_u}(g_1(t),t),\] 
		which vanishes at $0$ by \eqref{eq:lambdaCv}.
		In other words, the tangent vector to the singular curve of $C_1$ at $0$ is the vector $(0,1)$, 
		which is parallel to the null vector field of $C_1$ at $0$ (Lemma \ref{lemma:nullvectorCj}). 
		Therefore $0$ is a singular point of the second kind of $C_1$.}
	The case $j=2$ can be shown similarly.

	On the other hand, if $0$ is a first order $\bV_j$-ridge point $(j=1,2)$ of a frontal $f\colon(\R^2,0)\to(\R^3,0)$, 
	$0$ is a non-degenerate singular point of $C_j$ (Lemma \ref{lem:non-deg}), 
	in particular, $(\bV_j\rho_j)_v(0)\neq0$ holds. 
	Thus by the implicit function theorem, 
	there exists a curve $\beta_j(t)=(t,g_j(t))$ such that $g_j(0)=0$ and $\bV_j\rho_j(t,g_j(t))=0$. 
	The tangent vector of $(t,g_j(t))$ at $t=0$ is $(1,g_j'(0))$, where $g_j'(0)=-((\bV_j\rho_j)_u/(\bV_j\rho_j)_v)(0)$. 
	Since $(1,g_j'(0))$ is not parallel to the null direction $(0,1)$ of $C_j$, 
	the origin $0$ is a singular point of the first kind of $C_j$.
\end{proof}

By Proposition \ref{prop:sk}, $C_j$ ($j=1,2$) cannot be a $5/2$-cuspidal edge or a cuspidal cross cap  at $q$ 
when $q$ is a second order $\bV_j$-ridge point of $f$. 
In the following, we assume that $0$ is a singular point of the first kind of $C_j$. 
Then we give conditions that $C_j$ has a cuspidal cross cap at $0$ in terms of geometrical properties of the initial frontal.

\begin{thm}\label{thm:ccr}
	Let $f\colon(\R^2,0)\to(\R^3,0)$ be a frontal with pure-frontal singular point $0$. 
	Suppose that $0$ is a first order $\bV_j$-ridge point of $f$ 
	and a singular point of the first kind of $C_j$ $(j=1,2)$. 
	Then $C_j$ has a cuspidal cross cap at $0$ if and only if
	$r_c'\neq0$ and 
\[
\begin{aligned}
\hat{\rho}_j'\left(\kappa_t^2\kappa_\nu'+2\kappa_t\kappa_t' (\kappa_j-\kappa_\nu)+\frac{r_b'}{3}(\kappa_j-\kappa_\nu)^2\right)\\
\neq(\kappa_t^2+(\kappa_j-\kappa_\nu)^2)(\kappa_{j+1}-\kappa_j)(\kappa_\nu-\kappa_j)
\end{aligned}
\]
	hold at $0$, where $\hat{\rho}_j$ is the composition of $\rho_j$ and the singular curve $\beta(t)$ of $C_j$ 
	and $\hat{\rho}_j'$ is its derivative. 
\end{thm}
\begin{proof}
	We show the case for $j=1$. 
	Since $f$ is not a front, we may assume $f$ as in 
\begin{equation}\label{eq:normalform}
	f(u,v)=(u, u^2 a_2(u) + v^2/2, u^2 a_3(u) + v^2 c_2(u) + v^4 c_4(u) + v^5 c_5(u, v)),
\end{equation}
with $c_2(0)=0$ (cf. \cite[Proposition 3.9]{hs} and \cite[Proposition 2.1]{os}). 
Since $0$ is a singular point of $C_j$, $0$ is not a $5/2$-cuspidal edge of $f$. 
Thus $c_5$ as in \eqref{eq:normalform} satisfies $c_5(0,0)=0$ (cf. \cite{hks,hs}).

We note that, in this case, 
the singular set of $f$ is the $u$-axis and $\eta=\partial_v$ is a null vector field of $f$. 
Moreover, the functions in \eqref{eq:fundamentals} can be defined in the same way. 
Since $0$ is a singular point of the first kind of $C_j$, 
the singular curve $\beta(t)$ of $C_1$ can be represented as $\beta(t)=(t,g(t))$ with $g(0)=0$ 
by the proof of Proposition \ref{prop:sk}. 
It is well known (see \cite[Corollary 1.5]{fsuy}, for example) that, 
	with the notations we used so far, $C_1$ is a cuspidal crosscap at $0$
	if and only if $\psi_{C_1}(0)=0$ and $\psi_{C_1}'(0)\neq0$, where 
	\[
	\psi_{C_1}(t)=\det(\hat{\beta}_1'(t),\hat{\bm{e}}_1(t),d{\e_1}_{\beta(t)}(\bV_1)),
	\]
	$\hat{\beta}_1(t)=C_1(\beta(t))$, $\hat{\bm{e}}_1(t)=\bm{e}_1(\beta(t))$ and 
	$\bV_1$ is a null vector field of $C_1$. 
	Since $d\e_1(\bV_1)(0)=0$, $\psi_{C_1}(0)=0$ holds. 
	
	We calculate $\psi_{{C}_1}'$. 
	First we consider $\hat{\beta}_1'\times\hat{\e}_1$ at $0$. 
	Along $\beta(t)$, we have 
	\[\hat{\beta}_1'=(C_1)_u+(C_1)_vg'(t),\]
	and hence $\hat{\beta}_1'(0)=(C_1)_u(0)$ holds since $(C_1)_v=0$ at $0$. 
	Using Lemmas \ref{lem:weingarten} and \ref{lem:invariant}, it follows that 
	\[(C_1)_u=(1-\kappa_\nu\rho_1)f_u-\rho_1\kappa_th+(\rho_1)_u\nu
	=\rho_1((\kappa_1-\kappa_\nu)f_u-\kappa_th)+(\rho_1)_u\nu\]
	holds at $0$. 
	Moreover, $\hat{\e}_1$ can be written as 
	\[\hat{\e}_1=\frac{-\kappa_tf_u+(\kappa_\nu-\kappa_1)h}{\Delta_1}\]
	at $0$, where $\Delta_1=\sqrt{\kappa_t^2+(\kappa_1-\kappa_\nu)^2}$, 
	and hence we have 
	\[\hat{\beta}_1'\times\hat{\e}_1=-\frac{1}{\Delta_1}
	\left((\rho_1)_u((\kappa_\nu-\kappa_1)f_u+\kappa_th)+\rho_1\Delta_1^2\nu\right)\]
	at $0$. 
	
	We next consider the derivative of the curve $\alpha(t)=d{\e_1}_{\beta(t)}(\bV_1)$. 
	Since $\alpha(t)$ can be represented as 
	\[
	\alpha(t)=V_{11}(\beta(t))(\e_1)_u(\beta(t))+V_{12}(\beta(t))(\e_1)_v(\beta(t)),
	\] 
	we have 
	\[
	\begin{aligned}
	\alpha'(t)&=((V_{11})_u+(V_{11})_vg'(t))(\e_1)_u(\beta(t))\\&+V_{11}((\e_1)_{uu}(\beta(t))+(\e_1)_{uv}(\beta(t))g'(t))\\
	&+((V_{12})_u+(V_{12})_vg'(t))(\e_1)_v(\beta(t))\\&+V_{12}(\beta(t))((\e_1)_{uv}(\beta)(t)+(\e_1)_{vv}g'(t)).
	\end{aligned} 
\]
	Here $(\e_1)_v=0$ and $V_{11}=(V_{11})_u=0$ hold at $0$. 
	Hence it follows that 
	\begin{equation}\label{eq:dalphab}
	\alpha'=g'\left((V_{11})_v(\e_1)_u+V_{12}(\e_1)_{vv}\right)+V_{12}(\e_1)_{uv}
	\end{equation}
	holds at $0$.   
	
	Since $0$ is a pure-frontal singular point of $f$ and $f_v(u,0)=0$,
	$f_v(u,v)=vh(u,v)$ and $\nu_v(u,v)=v\nu_1(u,v)$, for smooth maps $h$ and $\nu_1$. 
	Thus we have $\wtil{E}_v(u,0)=\wtil{L}_v(u,0)=0$. On the other hand, from \eqref{eq:normalform}, it holds

\begin{equation}\label{eq:h}
	h(u,v)=(0,1,2c_2(u)+v^2\left(4c_4(u)+v\left(5c_5(u,v)+v(c_5)_v(u,v)\right)\right)).
\end{equation}	
Therefore we get $h_v(u,0)=h_{uv}(u,0)=0$ and consequently, $\wtil{M}_v(u,0)=\wtil{F}_v(u,0)=\wtil{G}_v(u,0)=0$.
	Thus we have
	$(\bx_1)_v=0$ at $0$ 
	since $(\kappa_1)_v(0,0)=0$. 
	We note that $(V_{11})_v=-\kappa_t$ and $V_{12}=\kappa_\nu-\kappa_1$ hold at $0$. 
	
	Since $\e_1=\bx_1/|\bx_1|$, we have 
	\[
	\begin{aligned}
	(\e_1)_u&=\frac{(\bx_1)_u}{|\bx_1|}-\e_1\frac{(|\bx_1|)_u}{|\bx_1|},\quad 
	(\e_1)_v=0,\\
	(\e_1)_{uv}&=(\e_1)_{vu}=\frac{(\bx_1)_{vu}}{|\bx_1|}-\e_1\left(\frac{(|\bx_1|)_v}{|\bx_1|}\right)_u,\\
	(\e_1)_{vv}&=\frac{(\bx_1)_{vv}}{|\bx_1|}-\e_1\left(\frac{(|\bx_1|)_v}{|\bx_1|}\right)_v
	\end{aligned}
	\] at $0$.
	By a direct calculation, we get
	\[(\bx_1)_u=-(\kappa_t'+\kappa_s(\kappa_\nu-\kappa_1))f_u+((\kappa_\nu'-(\kappa_1)_u)-\kappa_s\kappa_t)h-\kappa_1\kappa_t\nu\]
	at $0$. 
	Thus we see that
	\begin{equation}\label{eq:det1}
	\det(\hat{\beta}_1',\hat{\e}_1,(\e_1)_u)
	=\frac{1}{\Delta_1^2}\left((\rho_1)_u\left(\kappa_s\Delta_1^2+\kappa_t'(\kappa_\nu-\kappa_1)+\kappa_t((\kappa_1)_u-\kappa_\nu')\right)+\kappa_t\Delta_1^2\right)
	\end{equation}
	at $0$. 
	
	By $(\bx_1)_v(0,0)=0$, we obtain $\det(\hat{\beta}_1',\hat{\e}_1,(\bx_1)_v)=0$ at 0, 
	and hence $\det(\hat{\beta}_1',\hat{\e}_1,(\e_1)_v)=0$ at $0$. 
	Moreover, since $\wtil{E}_v, \wtil{F}_v, \wtil{G}_v, \wtil{L}_v$ and $\wtil{M}_v$ vanish on the $u$-axis, 
	then $\wtil{E}_{uv},\wtil{F}_{uv},\wtil{G}_{uv},\wtil{L}_{uv}$ and $\wtil{M}_{uv}$ also vansih on the $u$-axis. Using these facts and $f_{uuv}(0)=0,$ we have
	\[
	(\bx_1)_{uv}=-(\kappa_1)_{uv} h
	\]
	holds at $0$. 
	Therefore we get
	\[
	\begin{aligned}
	&\det(\hat{\beta}_1',\hat{\e}_1,(\bx_1)_{uv})=\frac{\rho_u \kappa_t (\kappa_1)_{uv}}{\Delta_1}
	\end{aligned}
\]
	at $0$. 
	Since $(|\bx_1|)_v=\inner{\bx_1}{(\bx_1)_v}/|\bx_1|=0$
	at $0$ and by \eqref{eq:kuv}, we have 
	\begin{equation}\label{eq:det2b}
	\det(\hat{\beta}_1',\hat{\e}_1,(\e_1)_{uv})
	=\frac{(\rho_1)_u\kappa_tr_c'(\kappa_1-\kappa_\nu)}{48\Delta_1^2\sqrt{\Gamma}}
	\end{equation}
	at $0$. 
	
	We next consider $(\bx_1)_{vv}$. 
	By using $\wtil{F}=\wtil{F}_v=\wtil{L}_v=\wtil{E}_v=(\kappa_1)_v=h_v=h_{vv}=0$, $\wtil{G}=\wtil{E}=1$ and $f_{uv}=0$ at $0$, we have 
	\begin{equation}\label{eq:xv1b}
	(\bx_1)_{vv}=-(\wtil{M}_{vv}-\kappa_1\wtil{F}_{vv})f_u-\wtil{M}f_{uvv}+(\wtil{L}_{vv}-(\kappa_1)_{vv}-\kappa_1\wtil{E}_{vv})h{+(\wtil{L}-\kappa_1\wtil{E})h_{vv}}
	\end{equation}
	at $0$. 
	By $f_v=vh$, \eqref{eq:h} and \eqref{eq:normalform}, we see that {$h_{vv}=\frac{r_b}{3}\nu$ and
		$f_{uvv}=h_u=\kappa_t\nu$ hold at $0$, which implies that $\wtil{F}_{vv}=0$}. 
	Moreover, we have the following.
	\begin{lem}\label{lem:lvvb}
		Under the above setting, we have 
		\[
		\wtil{L}_{vv}=\kappa_t'+\kappa_s\left(\kappa_\nu-\frac{r_b}{3}\right),\quad
		\wtil{M}_{vv}=\frac{r_b'}{3}+2\kappa_s \kappa_t
		\]
		at $0$. 
	\end{lem}
	\begin{proof}
	By definition of $\wtil{L}$ and $\wtil{M}$, and $f_{uv}=\nu_{uv}=h_v=h_{vv}=0$ at $0$, we have
	\[
	\wtil{L}_{vv}=-\inner{f_{uvv}}{\nu_u}-\inner{f_u}{\nu_{uvv}} \ \ \text{and}\ \  \wtil{M}_{vv}=-\inner{h}{\nu_{uvv}}
	\]
	at 0. By \eqref{eq:normalform}, $f_{uvv}=\kappa_t\nu$, which is orthogonal to $\nu$. Thus, $\wtil{L}_{vv}=-\inner{f_u}{\nu_{uvv}}$. On the other hand, since $\nu=\frac{\bar{\nu}}{|\bar{\nu}|}$ with $\bar{\nu}=f_u\times h$ and $\bar{\nu}_v(0)=0,$ it holds
	\[
	\nu_{uuv}=\frac{\bar{\nu}_{uvv}}{|\bar{\nu}|}-\frac{\bar{\nu}\inner{\bar{\nu}}{\bar{\nu}_{vvv}}}{|\bar{\nu}|^3}.
	\]
	By \eqref{eq:normalform} and direct calculations, 
\[
	\begin{aligned}
	\bar{\nu}=&\left(-2 u a_3 + 2u^2 c_2 a_2' + 
4 u^2 v^2 c_4 a_2'+5u^2v^3 c_5a_2'-u^2a_3'-v^2c_2'-v^4c_4'\right.\\
&\left.+u^2v^4a_2'(c_5)_v
+2u a_2(2c_2+4v^2c_4+5v^3 c_5+v^4(c_5)_v)-v^5 (c_5)_u,\right.\\
&\left.-2c_2-v^2(4c_4+v(5c_5+v (c_5)_v),1\right),
\end{aligned}
\]
where $'=d/du$. 
Thus we get 
	\[\bar{\nu}_{uvv}(0)=\left(16a_2(0)c_4(0)-2c_2''(0),-8c_4'(0),0\right),\quad 
	\bar{\nu}_{vvv}(0)=(0,0,0)\]
We note that $2a_2=\kappa_s$, $2a_3=\kappa_\nu$, $2c_2''=\kappa_s\kappa_\nu+\kappa_t'$, 
$24c_4=r_b$ and $24c'_4=r_b'+6\kappa_s\kappa_t$ hold at $0$ (cf. \cite[Page 508]{hs}). 
Hence we obtain
\[
\bar{\nu}_{uvv}=\left(-\kappa_t'+\kappa_s\left(\frac{r_b}{3}-\kappa_\nu\right),-\frac{r_b'}{3}-2\kappa_s\kappa_t,0\right) 
\ \ \text{and}\ \ \bar{\nu}_{vvv}=0
\]
at $0$. 
Now, since $f_u(0)=(1,0,0)$, $h(0)=(0,1,0)$, and $\nu(0)=(0,0,1)$, we have the result. 
\end{proof}
	We proceed calculations. 
	By Lemma \ref{lem:lvvb}, $(\bx_1)_{vv}$ as in \eqref{eq:xv1b} can be written as 
	\[
	\begin{aligned}
	(\bx_1)_{vv}
	&=-\left(\frac{r_b'}{3}+2\kappa_s\kappa_t\right)f_u
	+\left(\kappa_t'+\kappa_s\left(\kappa_\nu-\frac{r_b}{3}\right)-(\kappa_1)_{vv}\right)h\\
	&\hspace{7cm}+\left(\frac{r_b}{3}(\kappa_\nu-\kappa_1)-\kappa_t^2\right)\nu
		\end{aligned}
	\]
	at $0$. 
	Thus one can see that
	\[
	\begin{aligned}
	-\Delta_1&\det(\hat{\beta}_1',\hat{\e}_1,(\bx_1)_{vv})=-\Delta_1\inner{\hat{\beta}_1'\times\hat{\e}_1}{(\bx_1)_{vv}}\\
	=&(\rho_1)_u\left(\kappa_1 \left(2\kappa_s\kappa_t+\frac{r_b'}{3}\right)-\frac{\kappa_\nu r_b'}{3}
	-\kappa_t\left(-\kappa_t'+\kappa_s\left(\kappa_\nu+\frac{r_b}{3}\right)+(\kappa_1)_{vv}\right)\right)\\
	&\hspace{7cm}+\rho_1\Delta_1^2\left(\frac{r_b}{3}(\kappa_\nu-\kappa_1)-\kappa_t^2\right)
	\end{aligned}
	\]
	holds at $0$. 
	Hence we have 
\begin{equation}\label{eq:det3b}	
	\begin{aligned}
	&-\Delta_1^2\det(\hat{\beta}_1',\hat{\e}_1,(\e_1)_{vv})\\
	&=(\rho_1)_u\left(\kappa_1 \left(2\kappa_s\kappa_t+\frac{r_b'}{3}\right)-\frac{\kappa_\nu r_b'}{3}
	-\kappa_t\left(-\kappa_t'+\kappa_s\left(\kappa_\nu+\frac{r_b}{3}\right)+(\kappa_1)_{vv}\right)\right)\\
	&\hspace{7.5cm}+\rho_1\Delta_1^2\left(\frac{r_b}{3}(\kappa_\nu-\kappa_1)-\kappa_t^2\right)
	\end{aligned}
\end{equation}
	at $0$. 
	By \eqref{eq:det1}, and  
	\eqref{eq:det3b}, 
	 we have
	\begin{equation}\label{eq:det4b}
	\begin{aligned}
	&g'\det(\hat{\beta}_1',\hat{\e}_1,(V_{11})_v(\e_1)_u+V_{12}(\e_1)_{vv})\\
	&=g'((V_{11})_v\det(\hat{\beta}_1',\hat{\e}_1,(\e_1)_u)+V_{12}\det(\hat{\beta}_1',\hat{\e}_1,(\e_1)_{vv}))\\
	&=g'(-\kappa_t\det(\hat{\beta}_1',\hat{\e}_1,(\e_1)_u)
	+(\kappa_\nu-\kappa_1)\det(\hat{\beta}_1',\hat{\e}_1,(\e_1)_{vv}))\\
	&=\frac{g'}{\Delta_1^2}\left((\rho_1)_uA_1+\rho\Delta^2_1A_2\right)
	&
	\end{aligned}
	\end{equation}
	at $0$, where 
	\[\begin{aligned}
		A_1&=-\kappa_t (\bV_1\rho_1)_v \kappa_1^2+\kappa_s\kappa_t A_3+\frac{r_b'}{3}(\kappa_1-\kappa_\nu)^2+\kappa_t^2\kappa_\nu'+2\kappa_t\kappa_t'(\kappa_1-\kappa_\nu),\\ A_2&=\kappa_t^2(\kappa_\nu-2\kappa)-(\kappa_1-\kappa_\nu)^2\frac{r_b}{3}, \quad 
		A_3=-\Delta_1^2-(\kappa_\nu-\kappa_1)\left(2\kappa_1-\kappa_\nu-\frac{r_b}{3}\right),
	\end{aligned}
	\]
 and
  we use the identidy $(\bV_1\rho_1)_v \kappa_1^2=((\kappa_1)_{vv}(\kappa_1-\kappa_\nu)+\kappa_t (\kappa_1)_u)$ at $0$.
	Note that we can rewrite $A_2$ and $A_3$ as
	\[
	A_2=\left(\frac{\kappa_\nu r_b}{3}-\kappa_t^2\right)(2\kappa_1-\kappa_\nu)-\dfrac{\kappa_1^2r_b}{3} \ \ \text{and}\ \ 
	A_3=\kappa_1\left((\kappa_1-\kappa_\nu)-\frac{r_b}{3}\right)-\kappa_t^2+\frac{r_b \kappa_\nu}{3}.
	\]
	Now, using the identidies $\kappa_\nu+\frac{r_b}{3}=2 H=\kappa_1+\kappa_2$ 
	and $-\kappa_t^2+\frac{r_b \kappa_\nu}{3}=K=\kappa_1\kappa_2$ at $0$, 
	we get  $A_3=0$ and $A_2=\kappa_1(\kappa_1-\kappa_2)(\kappa_\nu-\kappa_1)$ at $0$.
	
	On the other hand, we see that 
	\[
	V_{12}\det(\hat{\beta}_1',\hat{\e}_1,(\e_1)_{uv})
	=-\frac{(\rho_1)_u\kappa_tr_c'(\kappa_1-\kappa_\nu)^2}{48\Delta_1^2\sqrt{\Gamma}}
	\]
	holds at $0$. 
	Since 
	\[
	(\bV_1\rho_1)_u=\frac{r_c'(\kappa_1-\kappa_\nu)^2}{48\kappa_1^2\sqrt{\Gamma}}
	\]
	holds at $0$ and $g'(0)=-\dfrac{(\bV_1\rho_1)_u}{(\bV_1\rho_1)_v}$, we have 
	\[
	-\dfrac{g'(\rho_1)_u\kappa_t(V_1\rho_1)_v \kappa_1^2}{\Delta_1^2}
	=\frac{(V_1\rho_1)_v(\rho_1)_u\kappa_t\kappa_1^2}{\Delta_1^2}
	=-V_{12}\det(\hat{\beta}_1',\hat{\e}_1,(\e_1)_{uv})
	\]
	at $0$. 
	Therefore we obtain 
	\[
	\begin{aligned}
	&\psi_{{C}_1}'=\det(\hat{\beta}_1',\hat{\e}_1,\alpha')\\
	&=\frac{g'(0)}{\Delta_1^2}\Big((\rho_1)_u\Big(\kappa_t^2\kappa_\nu'+2\kappa_t\kappa_t' (\kappa_1-\kappa_\nu)+\frac{r_b'}{3}(\kappa_1-\kappa_\nu)^2\Big)\\&
	\hspace{7cm}{+}\Delta_1^2(\kappa_1-\kappa_2)(\kappa_\nu-\kappa_1)\Big)
	\end{aligned}
\]
	at $0$ by \eqref{eq:dalphab}, \eqref{eq:det2b} and \eqref{eq:det4b}. 
	Since $\hat{\rho}_1'=(\rho_1)_u+(\rho_1)_vg'$ and $(\rho_1)_v=0$ at $0$, we have the assertion by \cite[Corollary 1.5]{fsuy}. 
\end{proof}

\begin{thm}\label{thm:pure}
	Let $f\colon(\R^2,0)\to(\R^3,0)$ be a frontal and $0$ a pure-frontal singular point of $f$. 
	Suppose that the secondary cuspidal curvature $r_c$ identically vanishes 
	along the singular curve $\gamma$ for $f$ through $0$, and $0$ is a first order $\bm{V}_j$-ridge point of $f$. 
	Then $\gamma$ is also a singular curve for $C_j$ $(j=1,2)$ consisting of pure-frontal singular points of $C_j$. 
	In addition, the Gaussian and the mean curvature of $C_j$ are bounded near $0$. 
	\end{thm}
\begin{proof}
	Let us take an adapted coordinate system $(u,v)$. 
	Then we show that $\psi_{C_j}$ $(j=1,2)$ vanishes along the singular curve of $C_j$. 
	By the definition of $\psi_{C_j}$, it is sufficient to show $d\e_j(\bV_j)=0$ along the singular curve for $C_j$. 
	We first remark that the $u$-axis is also a singular curve of $C_j$ because $(\bV_j\rho_j)=0$ along the $u$-axis by Proposition \ref{prop:ridge}. 
	Moreover, since $0$ is a first order $\bV_j$-ridge point, $((\bV_j\rho_j)_u,(\bV_j\rho_j)_v)\neq(0,0)$ at $0$. 
	Thus we get $S(f)=S(C_j)=\{v=0\}$. 
	
	By \eqref{eq:princ-vect}, 
	we have $d\e_j(\bm{V}_j)=(\kappa_\nu-\hat{\kappa}_j)(\e_j)_v$ along the $u$-axis, where $\hat{\kappa}_j(u)=\kappa_j(u,0)$. 
	By the proof of Proposition \ref{prop:frontalness}, it holds that 
	\begin{equation}\label{eq:ev}
		(\e_j)_v=\frac{(-1)^{j+1}r_c\kappa_t(\hat{\kappa}_j-\kappa_\nu)}{48\Delta_j^3\sqrt{\what{\Gamma}}}
		\left((\hat{\kappa}_j-\kappa_\nu)\hat{\gamma}'-\kappa_t\hat{h}\right)\quad 
		\left(\Delta_j=\sqrt{\kappa_t^2+(\hat{\kappa}_j-\kappa_\nu)^2}\right)
	\end{equation}
	along the $u$-axis, 
	where $\what{\Gamma}(u)=\Gamma(u,0)=H(u,0)^2-K(u,0)$. 
	Thus if $r_c$ vanishes identically along the $u$-axis, 
	then $d\e_j(\bm{V}_j)$ also vanishes. 
	This implies that $\psi_{{C}_j}$ vanishes along the $u$-axis, 
	and hence the $u$-axis consists of pure-frontal singular points for ${C}_j$. 
	Boundedness of the Gaussian and the mean curvature for $C_j$ follows from this result and 
	\cite[Proposition 3.8 and Theorem 3.9]{msuy}. 
\end{proof}

\begin{ex}\label{ex:helicoid}
Let $f\colon\R_{>0}\times\R\to\R^3$ be a $C^\infty$ map given by 
	\[f(u,v)=\left(-\cosh(\log u)\sin{v},\cosh(\log u)\cos{v},v\right),\]
	where $\R_{>0}=\{a\in\R\ |\ a>0\}$. 
	When we consider the ambient space as the Minkowski $3$-space $\R^3_1$ 
	with signature $(++-)$, 
	$f$ is a parametrization of the {\it maximal helicoid} (\cite{kobayashi}). 
	It is known that the set of singular points of $f$ is $S(f)=\{u=1\}$ and all singular points are fold singularities (cf. \cite{fruyy,fkkrsuyy}). 
	In fact, if we change parameter by $u=e^w$, then $f$ can be written as 
	\[f(w,v)=(-\cosh{w}\sin{v},\cosh{w}\cos{v},v).\]
	Further, $S(f)=\{w=0\}$ and $f$ satisfies $f(w,v)=f(-w,v)$.  
	
	We consider focal surfaces of $f$. 
	By direct calculations, we have 
	\[f_u=\frac{u^2-1}{2u^2}(-\sin{v},\cos{v},0),\quad 
	f_v=\left(-\frac{(1+u^2)\cos{v}}{2u},-\frac{(1+u^2)\sin{v}}{2u},1\right).\]
	Thus one can take a unit normal vector $\nu$ to $f$ as 
	\[\nu=\left(\frac{2u\cos{v}}{\sqrt{1+6u^2+u^4}},\frac{2u\sin{v}}{\sqrt{1+6u^2+u^4}},
	\frac{1+u^2}{\sqrt{1+6u^2+u^4}}\right).\]
	Using this $\nu$, we have principal curvatures $\kappa_1$ and $\kappa_2$ as follows:
	\[\kappa_1=-\frac{4u^2}{1+6u^2+u^4},\quad \kappa_2=\frac{4u^2}{1+6u^2+u^4}.\]
	These functions are of class $C^\infty$ even on $S(f)=\{u=1\}$. 
	Moreover, the reciprocals $\rho_1$, $\rho_2$ of them are also $C^\infty$ functions. 
	We note that $\im f$ is the subset of the right helicoid in $\R^3$. 
	So this can be considered as a minimal surface in $\R^3$ with fold singularities.  
	Setting $\delta=\sqrt{1+6u^2+u^4}$, focal surfaces $C_1$ and $C_2$ are written as 
	\[
	\begin{aligned}
		C_1&=\left(-\frac{{\delta}\cos{v}+(1+u^2)\sin{v}}{2u},-\frac{\delta\sin{v}-(1+u^2)\cos{v}}{2u},
		-\frac{\delta}{4}\left(1+\frac{1}{u^2}\right)+v\right),\\
		C_2&=\left(\frac{{\delta}\cos{v}-(1+u^2)\sin{v}}{2u},\frac{\delta\sin{v}+(1+u^2)\cos{v}}{2u},
		\frac{\delta}{4}\left(1+\frac{1}{u^2}\right)+v\right)
	\end{aligned}
	\]
	(see Figure \ref{fig:focals}). 
	
	We focus on $C_1$. 
	We note that the set $S(f)=\{u=1\}$ is also the set of singular points of $C_1$. 
	In fact, one can see that $(C_1)_u\times (C_1)_v=(u^4-1)\bm{a}(u,v)$ holds, 
	where $a(u,v)$ is a non-zero $\R^3$-valued function on $\R_{>0}\times\R$. 
	Moreover, since $C_1(1/u,v)=C_1(u,v)$ holds for any $(u,v)\in\R_{>0}\times\R$, 
	it might hold that {\it $C_1$ also has a fold singularity at $(1,v)$}. 
	(For $C_2$, we have the same property.)
	By direct calculations, a unit normal vector $\nu^{C_1}$ to $C_1$ can be taken as  
	\[\nu^{C_1}=\left(-\frac{(1+u^2)\cos{v}-\delta\sin{v}}{\sqrt{2}\delta},-\frac{(1+u^2)\sin{v}-\delta\cos{v}}{\sqrt{2}\delta},
	\frac{\sqrt{2}u}{\delta}\right).\]
Thus the Gaussian curvature $K^{C_1}$ and the mean curvature $H^{C_1}$ can be calculated as 
\[K^{C_1}=\frac{4u^4}{1+6u^2+u^4}>0,\quad 
H^{C_1}=-\frac{u}{\sqrt{2}(1+u^2)}(<0).\]
These are bounded $C^\infty$ functions on the source.
\begin{figure}[htbp]
	\centering
		\begin{tabular}{c}
			
			\begin{minipage}{0.3\hsize}
				\begin{center}
					\includegraphics[clip, width=2.75cm]{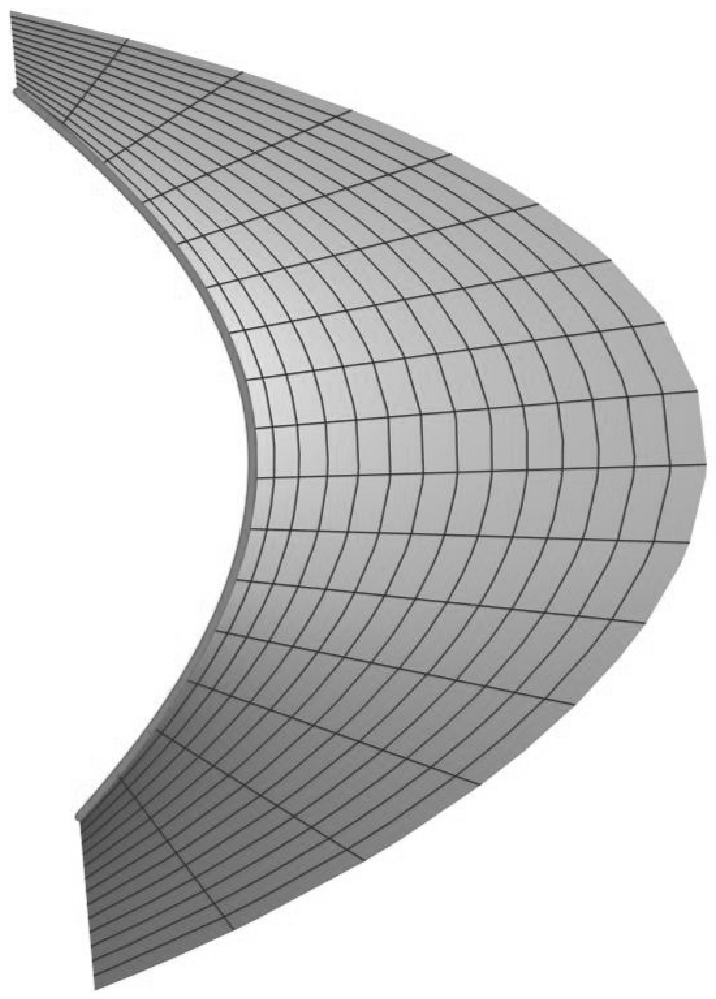}
				\end{center}
			\end{minipage}
			
			\begin{minipage}{0.3\hsize}
				\begin{center}
					\includegraphics[clip, width=3cm]{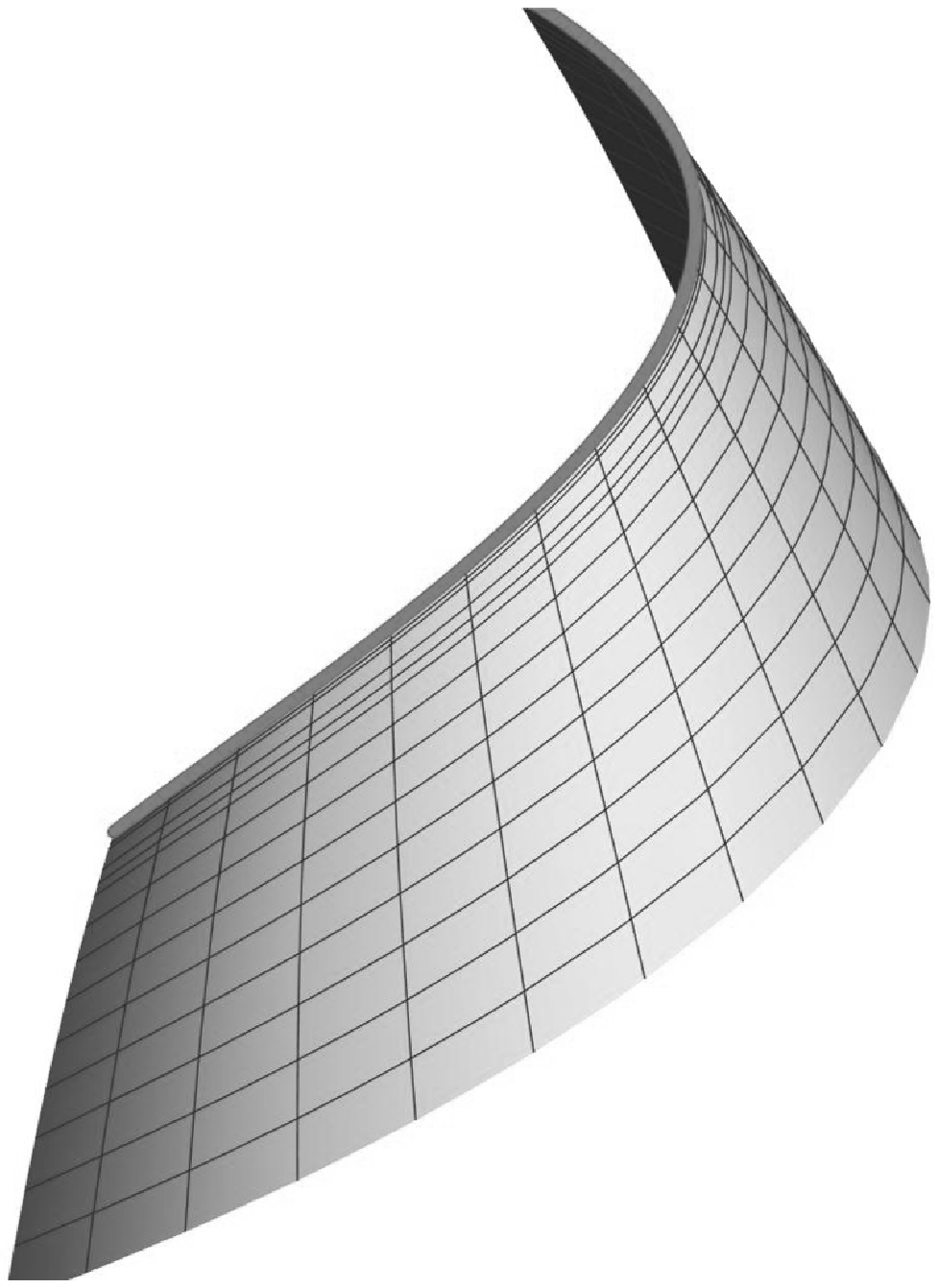}
				\end{center}
			\end{minipage}
			
			\begin{minipage}{0.3\hsize}
				\begin{center}
					\includegraphics[clip, width=3.5cm]{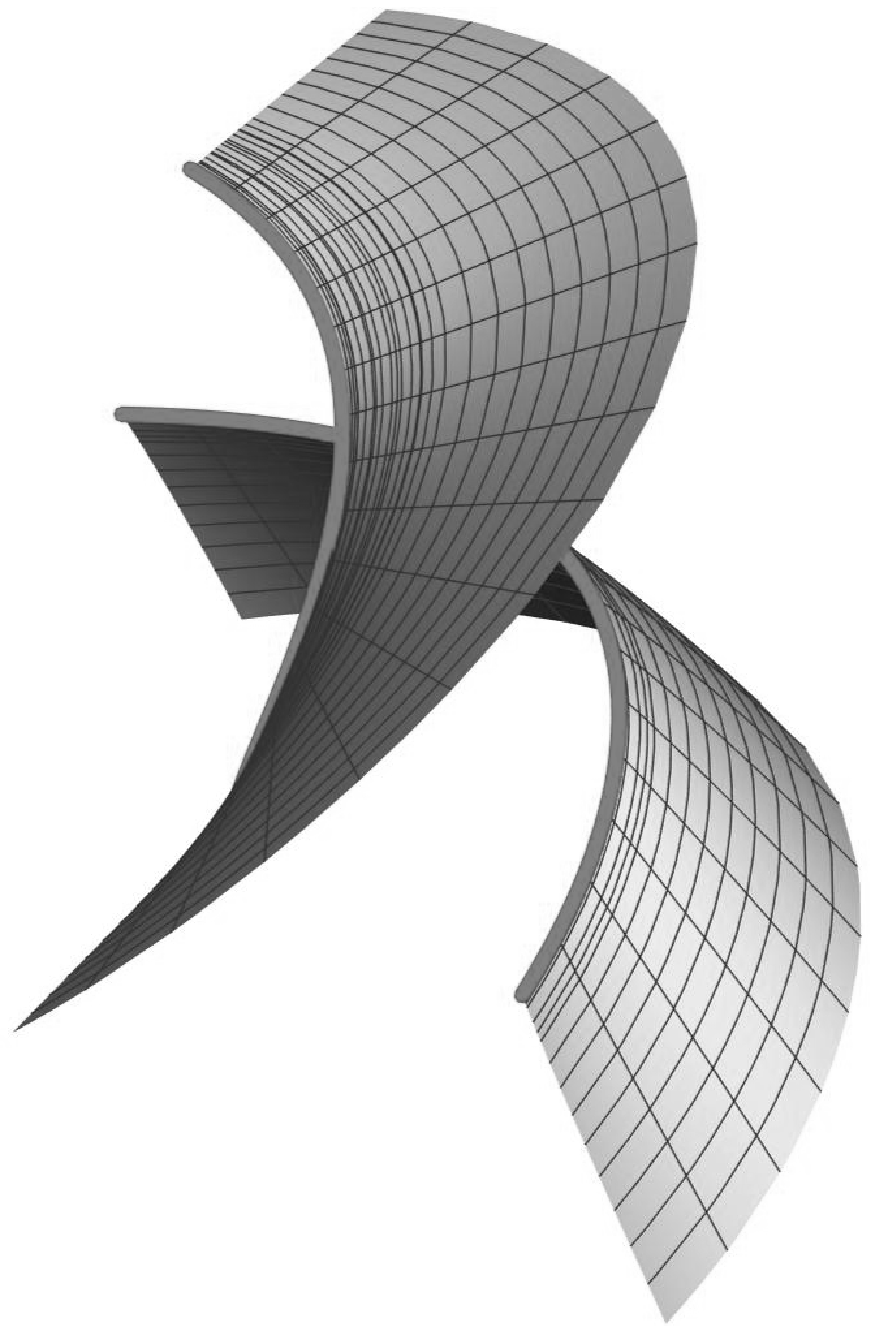}
				\end{center}
			\end{minipage}
		\end{tabular}
		\caption{The images of $f$ (left), $C_1$ (middle) and both of them (right) in Example \ref{ex:helicoid}.}
		\label{fig:focals}
\end{figure}
\end{ex}

\subsection{Curvatures of $C_j$ for $5/2$-cuspidal edges}
By Proposition \ref{prop:sing_mean}, when a frontal $f$ has a $5/2$-cuspidal edge at $0$, 
then focal surfaces $C_j$ $(j=1,2)$ are regular at $0$. 
Thus one can consider the Gaussian and mean curvature for $C_j$ at $0$. 
For the Gaussian curvature $K^{C_j}$ for $C_j$, we have the following assertion. 

\begin{thm}\label{thm:curvature}
	Let $f\colon(\R^2,0)\to(\R^3,0)$ be a frontal and $0$ a $5/2$-cuspidal edge of $f$. 
	Suppose that $\kappa_t\neq0$ at $0$. 
	 Then the Gaussian curvatures $K^{C_1}$ and $K^{C_2}$ of $C_1$ and $C_2$ are 
	written as 
	\[K^{C_j}=-\frac{\kappa_t^2\kappa_j^4}{({\kappa_t^2+(\kappa_\nu-\kappa_j)^2})^2}\]
	at $0$ $(j=1,2)$. 
	In particular, these are strictly negative at $0$.
\end{thm}
\begin{proof}
	Let us take an adapted coordinate system $(u,v)$ on $(\R^2,0)$. 
	Then we consider differentials of caustics $C_j$ ($j=1,2$). 
	By direct calculations, we have 
	\[(C_j)_u=(1-\rho_j\kappa_\nu)f_u-\rho_j\kappa_th+(\rho_j)_u\nu,\quad 
	(C_j)_v=(\rho_j)_v\nu\]
	at $0$. 
	We note that $(\rho_j)_v=-(\kappa_j)_v/\kappa_j^2\neq0$ at $0$ (cf. \eqref{eq:diff-kj}). 
	Thus the coefficients of the first fundamental form of $C_j$ are 
	\begin{equation*}\label{eq:fundamentalCj}
		E^{C_j}=\rho_j^2((\kappa_\nu-\kappa_j)^2+\kappa_t^2)+((\rho_j)_u)^2,\quad 
		F^{C_j}=(\rho_j)_u(\rho_j)_v,\quad 
		G^{C_j}=((\rho_j)_v)^2
	\end{equation*}
	at $0$. 
	Hence we get 
	\begin{equation}\label{eq:efgCj}
		E^{C_j}G^{C_j}-(F^{C_j})^2=((\rho_j)_v)^2\rho_j^2((\kappa_\nu-\kappa_j)^2+\kappa_t^2)
	\end{equation}
	at $0$. 
	We next investigate the coefficients of the second fundamental form of $C_j$. 
	By \eqref{eq:ev} and the above calculation, we have $\inner{(C_j)_v}{(\e_j)_v}=0$ at $0$. 
	Thus the quantity $N^{C_j}=-\inner{(C_j)_v}{(\e_j)_v}$ vanishes at $0$. 
	This implies that the Gaussian curvatures $K^{C_1}$ and $K^{C_2}$ are non-positive at $0$. 
	To obtain the explicit representation for $K^{C_j}$ ($j=1,2$), 
	we consider the quantity $M^{C_j}$ given by $M^{C_j}=-\inner{(C_j)_v}{(\e_j)_u}$. 
	Since $\e_j=\bx_j/|\bx_j|$, we have 
	$(\e_j)_u={(\bx_j)_u}{|\bx_j|^{-1}}+\bx_j(|\bx_j|^{-1})_u$. 
	Since $\bx_j\perp (C_j)_v$ hold, 
	it follows $\inner{(C_j)_v}{(\e_j)_u}=\inner{(C_j)_v}{(\bx_j)_u}|\bx_j|^{-1}$. 
	Thus we first consider $(\bx_j)_u$. 
	The map $\bx_j$ can be written as 
	\[\bx_j=-\kappa_tf_u+(\kappa_\nu-\kappa_j)h\]
	along the $u$-axis. 
	Therefore one can see that 
	\[(\bx_j)_u=-\kappa_t'f_u-\kappa_tf_{uu}+(\kappa_\nu'-(\kappa_j)_u)h+(\kappa_\nu-\kappa_j)h_u\]
	holds along the $u$-axis.  
	By \eqref{eq:frenet}, we may write $f_{uu}=\kappa_sh+\kappa_\nu\nu$ and $h_u=-\kappa_sf_u+\kappa_t\nu$ 
	along the $u$-axis. 
	Hence $(\bx_j)_u$ can be expressed as 
	\[(\bx_j)_u=-(\kappa_t'+\kappa_s(\kappa_\nu-\kappa_j))f_u
	+(\kappa_\nu'-(\kappa_j)_u-\kappa_s\kappa_t)h-\kappa_t\kappa_j\nu\]
	along the $u$-axis, in particular at $0$. 
	Thus it holds that 
	
	\begin{equation}\label{eq:MCj}
		M^{C_j}=-\inner{(C_j)_v}{(\e_j)_u}=\frac{(\rho_j)_v\kappa_t\kappa_j}{\sqrt{\kappa_t^2+(\kappa_\nu-\kappa_j)^2}}(\neq0)
	\end{equation}
	at $0$.  
	Thus the Gaussian curvature $K^{C_j}$ of $C_j$ is calculated as 
	\[
	\begin{aligned}
	K^{C_j}&=-\frac{(M^{C_j})^2}{E^{C_j}G^{C_j}-(F^{C_j})^2}=
	-\frac{\kappa_t^2\kappa_j^2}{\rho_j^2({\kappa_t^2+(\kappa_\nu-\kappa_j)^2})^2}\\
	&=-\frac{\kappa_t^2\kappa_j^4}{({\kappa_t^2+(\kappa_\nu-\kappa_j)^2})^2}
	\end{aligned}
\]
	at $0$ by \eqref{eq:efgCj} and \eqref{eq:MCj}. 
	This is the desired one.
	In particular, $K^{C_j}$ is strictly negative at $0$ by $\kappa_t\neq0$. 
\end{proof}
This theorem implies that a $5/2$-cuspidal edge corresponds to a {\it hyperbolic point} 
(cf. \cite[Page 12]{ifrt}) of the focal surfaces $C_j$ 
when $\kappa_t$ of the initial frontal does not vanish at that point. 

For the mean curvature,	we have the following. 
\begin{prop}\label{prop:Cj-mean}
	The mean curvature $H^{C_j}$ $(j=1,2)$ of $C_j$ for a frontal $f$ with a $5/2$-cuspidal edge $0$ 
	can be represented as 
	\[
	H^{C_j}=-\frac{\kappa_j(\kappa_s(\kappa_t^2+(\kappa_\nu-\kappa_j)^2)+\kappa_t'(\kappa_\nu-\kappa_j)-\kappa_t\kappa_\nu')}
	{2({\kappa_t^2+(\kappa_\nu-\kappa_j)^2})^{3/2}}
	\]
	at $0$. 
\end{prop}

\begin{proof}
	We take an adapted coordinate system $(u,v)$. 
	By the proof of Theorem \ref{thm:curvature}, we have 
	\[
	\begin{aligned}
		&L^{C_j}|\bx_j|=-\inner{(C_j)_u}{(\bx_j)_u}\\
		&=-\rho_j(\kappa_s(\kappa_t^2+(\kappa_\nu-\kappa_j)^2)
		{+}\kappa_t'(\kappa_\nu-\kappa_j)-\kappa_t(\kappa_\nu'-(\kappa_j)_u))
		+\kappa_t\kappa_j(\rho_j)_u
	\end{aligned}
\]
	at $0$. 
	Since $N^{C_j}=0$ at $0$ and relations $(\rho_j)_u=-(\kappa_j)_u\rho_j^2$ and $\kappa_j\rho_j=1$ hold, 
	we have 
	\[
	\begin{aligned}
		&E^{C_j}N^{C_j}-2F^{C_j}M^{C_j}+G^{C_j}L^{C_j}\\&=-2F^{C_j}M^{C_j}+G^{C_j}L^{C_j}\\
		&=-\frac{((\rho_j)_v)^2\rho_j\left(\kappa_s(\kappa_t^2
			+(\kappa_\nu-\kappa_j)^2){+}\kappa_t'(\kappa_\nu-\kappa_j)-\kappa_t\kappa_\nu'\right)}
		{\sqrt{\kappa_t^2+(\kappa_\nu-\kappa_j)^2}}
	\end{aligned}
\]
	at $0$. 
	By \eqref{eq:efgCj} and the above expression, we have the assertion. 
\end{proof}

\begin{ex}
		Let $f\colon(\R^2,0)\to(\R^3,0)$ be a map given by 
		\[f(u,v)=\left(u,u^2+\frac{v^2}{2},uv^2+\frac{v^5}{5}\right).\]
		Then one can see that $S(f)=\{v=0\}$ and $f$ has a $5/2$-cuspidal edge at $0$. 
		We can take a unit normal vector $\nu$ to $f$ as 
		\[\nu(u,v)=\frac{(2u(2u+v^3)-v^2,-2u-v^3,1)}{\sqrt{1+(2u+v^3)^2+(v^2-2u(2u+v^3))^2}}.\]
		We note that $\kappa_s=2$, $\kappa_\nu=0$, $\kappa_t=2\neq0$, $r_b=0$ and $r_c=72$ hold at $0$. 
		Thus principal curvatures $\kappa_1,\kappa_2$ take different values. 
		In particular, $\kappa_1=2$ and $\kappa_2=-2$ hold at $0$. 
		Thus the Gaussian curvature $K$ and the mean curvature $H$ of $f$ are 
		$K=-4<0$ and $H=0$ at $0$, respectively. 
		We consider the Gaussian curvatures and mean curvatures of $C_1$ and $C_2$. 
		By direct calculations, we have 
		\[K^{C_1}=K^{C_2}=-1<0\]
		at $0$. 
		On the other hand, 
		\[-\frac{\kappa_t^2\kappa_j^4}{({\kappa_t^2+(\kappa_\nu-\kappa_j)^2})^2}
		=-1\]
		holds at $0$ for $j=1,2$. 
		This verifies Theorem \ref{thm:curvature}. 
		Moreover, we have 
		\[H^{C_1}=-\frac{3}{2\sqrt{2}},\quad H^{C_2}=\frac{3}{2\sqrt{2}}\]
		at $0$ by direct calculations. 
		On the other hand, since $\kappa_t'=0$ and $\kappa_\nu'=-4$ hold at $0$, it follows that 
		\[
		\begin{aligned}
		{-}\frac{\kappa_1(\kappa_s(\kappa_t^2+(\kappa_\nu-\kappa_1)^2)+\kappa_t'(\kappa_\nu-\kappa_1)-\kappa_t\kappa_\nu')}
		{2({\kappa_t^2+(\kappa_\nu-\kappa_1)^2})^{3/2}}&=
		-\frac{3}{2\sqrt{2}},\\
		{-}\frac{\kappa_2(\kappa_s(\kappa_t^2+(\kappa_\nu-\kappa_2)^2)+\kappa_t'(\kappa_\nu-\kappa_2)-\kappa_t\kappa_\nu')}
		{2({\kappa_t^2+(\kappa_\nu-\kappa_2)^2})^{3/2}}&=
		\frac{3}{2\sqrt{2}}
		\end{aligned}
		\]
		hold at $0$. 
		This verifies that Proposition \ref{prop:Cj-mean} holds.
\end{ex}
%

\begin{ack}
The authors would like to thank Professors Luciana F. Martins and Kentaro Saji for fruitful discussions and comments. 
The second author thanks Professor Shoichi Fujimori for valuable comments. 
\end{ack}

\end{document}